\documentclass[11pt]{amsart}

\usepackage{amsmath, amsthm, amssymb, amsfonts, enumerate}
\usepackage[colorlinks=true,linkcolor=blue,urlcolor=blue]{hyperref}
\usepackage{color}
\usepackage{geometry}
\usepackage{mathtools}
\usepackage{graphicx}
\usepackage{bm}
\usepackage{verbatim}
\usepackage{enumitem}
\newtheorem{theorem}{Theorem}[section]
\newtheorem{remark}[theorem]{Remark}
\newtheorem{assumption}[theorem]{Assumption}

\newtheorem{proposition}[theorem]{Proposition}

\def \cC{{\mathcal C}}
\def \cD{{\mathcal D}}
\def \cF{{\mathcal F}}

\def \cL{{\mathcal L}}

\def \E{\mathsf{E}}
\def \P{\mathsf{P}}
\def \R{\mathbb{R}}

\def \N{\mathbb{N}}

\def \eps{\varepsilon}

\def \d{{\mathrm d}}
\def \e{{\mathrm e}}

\def \mathds {{\bf}}
\newcommand{\rom}[1]{\uppercase\expandafter{\romannumeral #1\relax}}

\title[A change of variable formula for optimal stopping]{A change of variable formula with applications to multi-dimensional optimal stopping problems}
\author[C.~Cai]{Cheng Cai}
\author[T.~De Angelis]{Tiziano De Angelis}
\address{C.~Cai: School of Mathematics, University of Leeds, Woodhouse Lane, LS2 9JT Leeds, UK.}
\email{\href{mailto:mmcca@leeds.ac.uk}{mmcca@leeds.ac.uk}}
\address{T.~De Angelis: School of Management and Economics, Dept.\ ESOMAS, University of Turin, C.so Unione Sovietica 218bis, 10134, Turin, ITALY; Collegio Carlo Alberto, Piazza Arbarello 8, 10122, Turin, ITALY.}
\email{\href{mailto:tiziano.deangelis@unito.it}{tiziano.deangelis@unito.it}}
\date{\today}

\numberwithin{equation}{section}

\begin{document}

\begin{abstract}
We derive a change of variable formula for $C^1$ functions $U:\R_+\times\R^m\to\R$ whose second order spatial derivatives may explode and not be integrable in the neighbourhood of a surface $b:\R_+\times\R^{m-1}\to \R$ that splits the state space into two sets $\cC$ and $\cD$. The formula is tailored for applications in problems of optimal stopping where it is generally very hard to control the second order derivatives of the value function near the optimal stopping boundary. Differently to other existing papers on similar topics we only require that the surface $b$ be monotonic in each variable and we formally obtain the same expression as the classical It\^o's formula. 
\end{abstract}

\maketitle

\section{Introduction}
The main aim of this paper is to provide a change of variable formula for a process $U(t,X_t)$ where $U:\R_+\times\R^m\to\R$ is a function and $X$ a stochastic process. Our setting is tailored for optimal stopping problems but the result is also of independent interest since it complements existing generalisations of It\^o's formula. We could think of $U$ as the value function of an optimal stopping problem whose underlying stochastic process is a suitable multi-dimensional c\`adl\`ag semi-martingale $X$. With this in mind we divide the state space $\R_+\times\R^m$ into two subsets $\cC$ and $\cD$, whose boundary $\partial\cC$ would correspond to the optimal stopping boundary. Our focus is on obtaining a formula that resembles the classical It\^o's formula and does not involve either local times or the quadratic covariation between the underlying process $X$ and the spatial gradient $\nabla U(t,X)$. This is important, for example, when deriving the dynamics of hedging portfolios for American options on multiple assets or integral equations for optimal stopping boundaries (in the spirit of numerous examples in the book by Peskir and Shiryaev \cite{PSbook}). Since we want to avoid using local times and quadratic covariation, we do require that the spatial gradient $\nabla U$ be a continuous function. However, we require minimal regularity on the second order spatial derivatives of $U$ near the boundary $\partial\cC$ and very mild monotonicity properties of the boundary itself. Our assumptions will be shown to hold naturally in a very broad class of optimal stopping problems for which existing generalisations of It\^o's formula are either technically more involved than ours or not applicable (see Section \ref{sec:appl}). A key difficulty for the application of existing formulae in the context of optimal stopping for {\em multi-dimensional} semi-martingales ($m\ge 2$) is that neither the value function nor the optimal stopping boundary of the problem are known explicitly. Therefore, two non-trivial technical problems arise in practice: 
\begin{itemize}
\item[(i)] it is often very difficult to prove sufficient regularity of the second-order spatial derivatives of the value function up to the optimal boundary; 
\item[(ii)] in some cases it is required that by composing the optimal stopping boundary with the underlying stochastic process (in a sense clarified below) one should obtain a semi-martingale. Due to limited knowledge on the smoothness of the boundary, such a semi-martingale assumption is often difficult to verify in practice. \end{itemize}
As we will clarify in the rest of this Introduction and in Section \ref{sec:appl}, our paper circumvents those hurdles and provides a viable tool for applied problems.

We now review some of the main results in the field but without the ambition to give a full account of the existing literature, which is vast and branches out in several specialised directions. In order to avoid confusion with our own setting, below we use $F$ to denote the function to which the change of variable formula is applied in the literature that we discuss.

Various change of variable formulae have been developed that do not even require continuity of first order spatial derivatives of $F$. Perhaps the best known one is the so-called It\^o-Tanaka-Meyer formula (see, e.g., \cite[Thm.IV.7.70]{protterbook}) which applies to functions $F:\R\to\R$ that are a difference of convex functions (see also \cite[Sec.\ 3]{azema1998} for an extension to $F(t,X_t)$ with $X$ a one-dimensional Brownian motion). Relaxing the assumption of convexity is generally difficult but a number of results are known in the literature. An early work in this direction is the one by Bouleau and Yor \cite{bouleau1981} who establish a formula for functions $F:\R\to\R$ which are absolutely continuous with locally bounded first order derivative and for a fairly broad class of c\`adl\`ag semi-martingales. The key idea in that work is that the semi-martingale local time defines a measure on $\R$ via the mapping $a\mapsto L^a_t$ (see, e.g., \cite[Thm.\ IV.7.77]{protterbook} and the subsequent corollary for details). F{\"o}llmer and Protter \cite{follmer2000ito} generalise those results to functions $F:\R^d\to\R$ whose first order partial derivatives exist in the weak sense as functions in $L^2$ and the underlying process is a $d$-dimensional Brownian motion. Analogous results in the one-dimensional case had been previously obtained by F{\"o}llmer, Protter and Shiryaev in \cite{follmer1995quadratic} (see also Bardina and Jolis \cite{bardina1997extension} for time-space extensions in the case of one-dimensional diffusions with suitable transition density). Those works shift the focus from the use of semi-martingale local times (as in Bouleau and Yor \cite{bouleau1981}) to the use of quadratic covariation of $\nabla F(X)$ and $X$. Quadratic covariation appears also in work by Russo and Vallois \cite{russo1996ito}, who require continuous differentiability of the function $F:\R^d\to \R$ but develop change of variable formulae for more general processes than just semi-martingales, thanks to notions of forward and backward integrals that they introduced in earlier papers (see also subsequent results by Errami, Russo and Vallois \cite{errami2002ito}). Further results based on quadratic covariation of $\nabla F(X)$ and $X$ are established by Moret and Nualart \cite{moret2001generalization} when $F$ belongs to the Sobolev class $W_{\ell oc}^{1,p}(\R^d)$ and $X$ is a non-degenerate martingale, using Malliavin calculus techniques. In the case of diffusions associated to uniformly elliptic operators in divergence form Rozkosz \cite{rozkosz1996stochastic} establishes a change of variable formula for functions $F$ in the class $W_{\ell oc}^{1,p}(\R^d)$, for $p> 2\wedge d$, via Stratonovich integrals. 

The focus on properties of local times of semi-martingales is central in works by Peskir \cite{peskir2005change} and \cite{peskir2007change}, which are close in spirit to our paper (see also \cite{ghomrasni2004local} for further results and links to other generalisations of It\^o's formula). In particular, in \cite{peskir2005change} Peskir studies a change of variable formula for processes $F(t,X_t)$ where $X$ is a continuous semi-martingale, $F:\R_+\times\R\to \R$ is such that $F\in C^{1,2}$ separately in the closure of two sets $\cC$ and $\cD$, with $\R_+\times \R=\cC\cup\cD$ and the sets are separated by the graph of a continuous function $b:\R_+\to \R$ of bounded variation. Spatial derivatives of $F$ need not be continuous across the boundary of the two sets $\partial\cC=\partial\cD$, which leads Peskir to consider the local time of $X$ along the curve $t\mapsto b(t)$. The $C^{1,2}$ requirement on $F$ can be weakened to hold only in the interior of the sets $\cC$ and $\cD$, separately, if $X$ is a continuous diffusion (see \cite[Sec.\ 3]{peskir2005change}).
In his other paper \cite{peskir2007change}, Peskir extends the result to multi-dimensional, possibly discontinuous semi-martingales $X\in\R^d$ and in this case the sets $\cC$ and $\cD$ are separated by the graph of a function $b:\R_+\times \R^{d-1}\to\R$ that is continuous and such that the process $b^X:= b(t,X^1,\ldots X^{d-1})$ is a semi-martingale. These assumptions on $b^X$ may be hard to verify directly in applications to optimal stopping, because the boundary $b$ is not given explicitly, and it was one of the main motivations for our own paper. Elworthy, Truman and Zhao \cite{elworthy2007generalized} also obtain change of variable formulae for time-space processes where the spatial component is a one-dimensional semi-martingale (for an extension to two-dimensional diffusions see \cite{feng2007generalized}); they require left-derivatives in time and space of the function $F$ to have bounded variation. 

Eisenbaum \cite{eisenbaum2006local} developes change of variable formulae for multi-dimensional L\'evy processes when first order partial derivatives of the function $F$ exist and are integrable, without further assumptions on second order derivatives. She relies on a suitable notion of integrals with respect to local time $(a,t)\mapsto L^a_t$, understood as integrator in both variables, and connects her results to all the papers we mentioned so far (see also \cite{eisenbaum2000integration} and \cite{eisenbaum2001ito} for earlier closely related work by the same author). More recently, Wilson \cite{W18} also studied integrals with respect to local time as a map $(a,t)\mapsto L^a_t$ (building upon ideas from \cite{eisenbaum2006local} and \cite{ghomrasni2004local}). 
He then uses such integrals in \cite{W19} to derive a change of variable formula for functions $F:\R_+\times\R^2\to \R$ when the underlying process is a two-dimensional jump diffusion process whose jumps are of bounded variation and with no diffusive part in the second component. Wilson's assumptions on $F$ are in the same spirit as those by Eisenbaum but his change of variable formula draws on \cite{peskir2005change} and \cite{peskir2007change}. However, \cite{W19} requires that either the boundary $b:\R_+\times \R\to\R$ be Lipschitz continuous or $b^X:=b(t,X^2)$ be of bounded variation. Both assumptions are generally difficult to check in applications to optimal stopping. Finally, under the assumption that smooth-fit holds and with an analogue of our Assumption \ref{ass:1-2} in place, \cite{W19} obtains a generalisation of It\^o's formula without requiring $b^X$ of bounded variation (but still requiring $X^2$ of bounded variation). 

It is worth mentioning that a number of interesting results on generalisations of It\^o's formula developed in the early 2000s are collected in the book \cite{LNP}. There we find for example work by Kyprianou and Surya \cite{Kyprianou} on a change of variable formula with local times on curves, for one-dimensional L\'evy processes of bounded variation. Some of the work by Eisenbaum, Peskir, Russo and Vallois are also contained therein.

In the theory of stochastic control the most widely used extensions of It\^o's formula for time-space diffusion processes (generally admitting smooth transition density), require $F\in W^{1,2,p}_{\ell oc}(\R_+\times\R^m)$ for $p>1$ sufficiently large to also guarantee that the spatial gradient $\nabla F$ is continuous thanks to Sobolev embedding (see, e.g., \cite[Ch.\ 2.8]{bensoussan}, \cite[Ch.\ 2 Sec.\ 10]{krylov2008controlled} or \cite[Ch.\ 8]{fleming}). While our proof is inspired by those results, we remark that our function $U$ does not belong to the Sobolev class $W^{1,2,p}_{\ell oc}(\R_+\times\R^m)$ because we do not require integrability of second order spatial derivatives   in neighbourhoods of the boundary $\partial\cC$. 
In the context of applications to optimal stopping it is also worth mentioning the work by Alsmeyer and Jaeger \cite{alsmeyer2005useful}. They prove a change of variable formula for functions $F:\R^{d+1}\to\R$ that are continuously differentiable and whose derivative in its first variable (denoted $D_{x_0}F$) is absolutely continuous as a map $z\mapsto D_{x_{0}}F(z,x_1,\ldots x_d)$ for all $(x_1,\ldots x_d)$ fixed. Differently from our set-up their result applies for processes $X=(M,V^1,\ldots V^d)$ where $M$ is a continuous semimartingale and $(V^1,\ldots V^d)$ is a continuous process of locally bounded variation.  

The paper is organised as follows. In Section \ref{sec:results} we present our framework and state our change of variable formula. In Section \ref{sec:appl} we discuss the applicability of our result in optimal stopping problems for multidimensional processes. In Section \ref{sec:proof} we prove our change of variable formula.

\section{Setting and main result}\label{sec:results}
On a filtered probability space $(\Omega,\cF,(\cF_t)_{t\ge 0},\P)$ we consider a $m$-dimensional Brownian motion $\bm{B}:=(B^1_t,\ldots B^m_t)_{t\ge 0}$ and denote by $\bm{X}:=(X^{1},...,X^{m})$ a solution in $\R^m$ of the stochastic differential equation (SDE): for $i=1,\ldots m$,
\begin{equation}
\label{eq:X-sde}
\d X^{i}_t=\alpha^{i}(t,\bm{X}_{t-})\d t+\sum_{j=1}^{m}\sigma^{ij}(t,\bm{X}_{t-})\d B^{j}_t+\gamma^i(t,\bm X_{t-})\d A^i_t, \quad X^{i}_0=x_i,
\end{equation}
where $\bm{A}=(A^{1},...,A^{m})$ is a c\`adl\`ag process of bounded variation with $\bm{A}_0=\bm 0$. 
Here we use boldface letters to indicate vectors 
and denote 
\[
\beta^{ij}(t,\bm{x}):=\sum_{k=1}^{m}\sigma^{ik}(t,\bm{x})\sigma^{jk}(t,\bm{x})
\]
and $f_{x_i}=\frac{\partial f}{\partial x_i}$, $f_{x_i x_j}=\frac{\partial^2 f}{\partial x_i\partial x_j}$ for all $i,j=1,\ldots m$. The coefficients of the SDE are assumed to be measurable and, for the sake of concreteness, we also assume for all $t\ge 0$ that
\[
\int_0^t\sum_{i=1}^m\big|\gamma^i(s,\bm X_{s-})\big|\d |A^i|_s+\int_0^t\Big( \sum_{i=1}^m\big|\alpha^{i}(s,\bm{X}_s)\big| + \sum_{i,j=1}^{m}\big|\sigma^{ij}(s,\bm{X}_s)\big|^2\Big)\d s<\infty,\qquad \P\text{-a.s.},
\]
where we denote by $|A^i|_s$ the total variation process associated to $A^i$.

We divide the state-space into two subsets, i.e., $\R_+\times\R^m=\cC\cup\cD$, with $\cC$ open and $\cD$ closed. We further assume that such subsets can be described in terms of a surface $b_1:\R_+\times \R^{m-1}\to \R$ as
\begin{align}
&\cC=\{(t,\bm{x})\in\R_+\times\R^m\, :\, x_1>b_1(t,x_2,...x_m)\},\label{eq:setC}\\
&\cD=\{(t,\bm{x})\in\R_+\times\R^m\, :\, x_1\le b_1(t,x_2,...x_m)\}\label{eq:setD}.
\end{align}

The main aim of the paper is to prove a change of variable formula for functions $U:\R_+\times\R^m\to \R$ whose second order spatial derivatives may explode along the boundary $\partial\cC$ arbitrarily fast. 
\begin{theorem}\label{thm:main}
Assume the following:
\begin{enumerate}[label=\textbf{A.\arabic*}]
\item
\label{ass:1-1}
The coefficients $\beta^{ij}$ are locally Lipschitz and $\P((t,\bm{X}_{t-})\in\partial \cC)=0$ for a.e.\ $t\ge 0$; 
\item
\label{ass:1-2}
A function $U:\R_+\times\R^m\to \R$ is such that $U\in C^1(\R_+\times\R^m)$ with $U\in C^{1,2}(\cC)\cap C^{1,2}(\cD)$. Moreover, for any compact subset $K\subset \R_+\times\R^m$ the function 
\begin{align}\label{eq:L}
L(t,\bm x):=\sum_{i,j=1}^m \beta^{ij}(t,\bm{x})U_{x_i x_j}(t,\bm{x})
\end{align} 
is bounded for $(t,\bm{x})\in K\setminus \partial\cC$. That is, for any compact $K$ there exists $c_K$ such that 
\begin{align}\label{eq:BL}
\sup_{(t,\bm x)\in K\setminus\partial\cC}|L(t,\bm x)|\le c_K;
\end{align} 
\item
\label{ass:1-3}
The mappings $x_i\mapsto b_1(t,x_2,\ldots,x_m)$, $i=2,\ldots m$, and $t\mapsto b_1(t,x_2,\ldots,x_m)$ are monotonic. 
\end{enumerate}

Then, we have the change of variable formula: 

\begin{equation}
\label{eq:Ito-0}
\begin{aligned}
&U(t,\bm{X}_{t})=U(0,\bm{x})\\
&\quad+\int_0^{t} \Big[\Big(U_t+\sum_{i=1}^m \alpha^i U_{x_i}\Big)(u,\bm{X}_{u-})+\tfrac{1}{2}\sum_{i,j=1}^m\mathds{1}_{\{(u,\bm X_{u-})\notin\partial\cC\}}\big(\beta^{ij}U_{x_i x_j}\big)(u,\bm{X}_{u-})\Big]\d u \\
&\quad+\sum_{i=1}^m\int_0^{t}\big(\gamma^i U_{x_i}\big)(u,\bm{X}_{u-})\d A^{c,i}_u+\sum_{u\le t} \Big(U(u,\bm X_u)-U(u,\bm X_{u-})\Big)\\
&\quad+\sum_{i,j=1}^{m}\int_{0}^{t}U_{x_i}(u,\bm{X}_{u-})\sigma^{ij}(u,\bm{X}_{u-}) \d B^{j}_u,\qquad\text{for $t\in [0,\infty)$, $\P$-a.s.,}
\end{aligned}
\end{equation}
where we used the decomposition $A^{i}_t=A^{c,i}_t+\sum_{s\le t}\Delta A^i_s$ with $A^{c,i}$ the continuous part of the process $A^i$.
\end{theorem}

Since the jumps of the process $\bm X$ only arise from the bounded variation process $\bm A$, the expression for the jump terms in \eqref{eq:Ito-0} is equivalent to the usual one found in textbooks:
\begin{align*}
&\sum_{i=1}^m\int_0^{t}\big(\gamma^i U_{x_i}\big)(u,\bm{X}_{u-})\d A^{c,i}_u+\sum_{u\le t} \Big(U(u,\bm X_u)-U(u,\bm X_{u-})\Big)\\
&=\sum_{i=1}^m\int_0^{t} \big(\gamma^i U_{x_i}\big)(u,\bm{X}_{u-})\d A^{i}_u+\sum_{u\le t}\Big(U(u,\bm X_u)-U(u,\bm X_{u-})-\sum_{i=1}^m \big(\gamma^i U_{x_i}\big)(u,\bm X_{u-})\Delta A^{i}_u\Big).
\end{align*}  

Assumption \ref{ass:1-2} says that the derivatives $U_{x_i x_j}$ are continuous separately in $\cC$ and in the interior of $\cD$. Moreover, the condition $U\in C^{1,2}(\cD)$ means that, for each $i,j=1,\ldots m$, the function $U_{x_i x_j}$ admits a continuous extension from the interior of $\cD$ to its boundary $\partial \cD=\partial \cC$. On the contrary, $U_{x_i x_j}$ need not admit a continuous extension from $\cC$ to $\partial\cC$ for $i,j=1,\ldots m$. 
In general boundedness of the function $L$ in \eqref{eq:L} is not sufficient for the boundedness of all second order spatial derivatives. 

The need to have some control over the function $L$ in \eqref{eq:L} was already indicated by Peskir in \cite[Thm.\ 3.1]{peskir2005change} (see the condition in Eq.\ (3.26) therein) in the case when the boundary $b$ is a continuous function of bounded variation only depending on time and $X$ is a one-dimensional diffusion process. Peskir et al.\ \cite[Thm.\ 19]{ernst} also employ a condition similar to \eqref{eq:BL} to obtain Dynkin's formula (rather than It\^o's formula) for a two-dimensional diffusion. Their proof requires different arguments to ours as they need convexity/concavity of their function $U$ and use estimates on the expected value of local times. 

\begin{remark}[Degenerate processes]
It is intuitively clear and it can be easily seen from the proof of the theorem that if the $i$-th coordinate of the process $\bm X$ is of bounded variation (i.e., $\sigma^{ij}\equiv 0$ for all $j=1,\ldots m$) it is {\em not} necessary to require existence of the second order partial derivatives $U_{x_i x_j}$ for $j=1,\ldots m$ in Assumption \ref{ass:1-2}.
\end{remark}

\begin{remark}[Assumptions on the boundary]
Assumption \ref{ass:1-3} is much easier to verify in applications to multi-dimensional optimal stopping problems than the assumption on the boundary $\partial\cC$ made in \cite{peskir2007change} (and more recently in \cite{W19} but only for two dimensional processes). In \cite{peskir2007change}, $\bm{X}$ is a general semi-martingale and the process $b^X_t=b(t,X^2_t,\ldots X^{m}_t)$ must also be a semi-martingale (with $b$ continuous). That is not true in general if only monotonicity of the boundary is known.  Of course, we are able to allow for much less stringent conditions on the boundary because, differently to \cite{peskir2007change}, our focus {\em is not} on the role of local times on surfaces and we assume continuous differentiability of the function $U$.
\end{remark}

\begin{remark}[Reflecting diffusions]
We chose to state our theorem including the bounded variation process $\bm A$ in the dynamics \eqref{eq:X-sde} because we have in mind applications to problems for reflecting diffusions and applications in singular stochastic control. In those cases, the condition $\P((t,\bm X_{t-})\in\partial\cC)=0$ for a.e.\ $t\ge 0$ in Assumption \ref{ass:1-1} is generally satisfied by Skorokhod's construction of reflecting diffusions. 
\end{remark}

\section{Applications in optimal stopping}\label{sec:appl}
Our main motivation for the development of a change of variable formula of the kind in Theorem \ref{thm:main} is its applicability in optimal stopping problems. In particular, when the underlying process is multi-dimensional and diffusive, existing change of variable formulae are difficult to apply because they require regularity conditions on the value function of the problem and on the associated optimal stopping boundary that are hard to check in practice. Here we illustrate some advantages of our formula for this type of applications. There are three subsections. In the first one we provide a general problem formulation and we discuss our specific assumptions in this context. In the second one we illustrate in detail four examples from the literature on optimal stopping (and singular control). In those examples an application of existing change of variable formulae is not possible whereas our formula can be applied directly. In the final subsection we formulate sufficient conditions on the problem data, in a general setting, that guarantee monotonicity of the optimal boundary as in Assumption \ref{ass:1-3}.

\subsection{A general optimal stopping framework}\label{subs:framework}
Letting $G:\R_+\times\R^m\to\R$ be a measurable function and $s\mapsto\Pi^t_s(\bm X)$ an additive functional of the process $(s,\bm X_s)_{s\ge t}$, one is often interested in problems of the type
\begin{align}\label{eq:OS}
U(t,\bm x)=\sup_{t\le \tau\le T}\E_{t,\bm x}\Big[\e^{-\Pi^t_\tau (\bm X)} G(\tau,\bm X_\tau)\Big],
\end{align}
where $T\in(0,\infty]$ is a fixed horizon, $t\in[0,T]$, the supremum is taken over stopping times of the underlying filtration $(\cF_t)$ and the expectation $\E_{t,\bm x}$ is with respect to the measure $\P_{t,\bm x}(\,\cdot\,):=\P(\,\cdot\,|\bm X_t=\bm x)$. In most examples the additive functional $\Pi^t$ arises from a discount rate, i.e.,
\begin{align}\label{eq:lambda}
\Pi^t_s(\bm X)=\int_t^s r(u,\bm X_{u-})\d u,
\end{align}
for some measurable functions $r:\R_+\times\R^m\to \R$. However, there are examples in which $\Pi^t$ may take other forms as, e.g., that of a local time of the process $\bm X$ (see, e.g., \cite{de2017dividend}).

Under a set of fairly mild assumptions, it is known that an optimal stopping time for the problem above exists and it takes the form (see, e.g., \cite{PSbook})
\begin{align}\label{eq:taustar}
\tau_*=\inf\{s\in[t,T]: U(s,\bm X_s)=G(s,\bm X_s)\}.
\end{align}
From this stems the interest for the study of the so-called continuation and stopping sets, denoted by $\cC$ and $\cD$, respectively, and defined as
\[
\cC=\{(t,\bm x): U(t,\bm x)>G(t,\bm x)\}\quad\text{and}\quad\cD=\{(t,\bm x): U(t,\bm x)=G(t,\bm x)\}.
\] 
In particular, parametrisations of the continuation and stopping sets as those presented in \eqref{eq:setC} and \eqref{eq:setD} are widely studied in the literature as they often enable a detailed theoretical analysis of the problem at hand. 

Together with the probabilistic results on optimality of $\tau_*$ and the so-called super-harmonic property of the value function (see \cite{PSbook}) there is also an analytic formulation of problem \eqref{eq:OS}, in terms of a free boundary problem. For simplicity let us take $\gamma^i\equiv 0$ in \eqref{eq:X-sde} and $\Pi^t$ as in \eqref{eq:lambda}. Then the free boundary problem solved by the value function reads
\begin{equation}
\begin{aligned}
&U_t+\tfrac{1}{2}\sum_{i,j}\beta^{ij}U_{x_i x_j}+\sum_i\alpha^i U_{x_i}-rU=0,\quad\text{in $\cC$},\\
&U_t+\tfrac{1}{2}\sum_{i,j}\beta^{ij}U_{x_i x_j}+\sum_i\alpha^i U_{x_i}-rU\le 0,\quad\text{in $\cD$},
\end{aligned}
\end{equation}
with terminal condition $U(T,\bm x)=G(T,\bm x)$.  It is clear that $U=G$ on $\cD$. If for example $G\in C^{1,2}(\cD)$, then $U$ inherits such regularity in $\cD$ and we have
\[
U_t+\tfrac{1}{2}\sum_{i,j}\beta^{ij}U_{x_i x_j}+\sum_i\alpha^i U_{x_i}-rU=G_t+\tfrac{1}{2}\sum_{i,j}\beta^{ij}G_{x_i x_j}+\sum_i\alpha^i G_{x_i}-rG,\quad\text{ in $\cD$}. 
\]
So, by the free boundary formulation we see that the function $L$ from Assumption \ref{ass:1-2} reads
\begin{align}\label{eq:LLL}
L(t,\bm x)=
\left\{
\begin{array}{ll}
2(rU-\sum_i\alpha^i U_{x_i}-U_t)(t,\bm x), & (t,\bm x)\in\cC,\\[+5pt]
\sum_{i,j}\beta^{ij}(t,\bm x) G_{x_i x_j}(t,\bm x), & (t,\bm x)\in\cD.
\end{array}
\right.
\end{align}

It is possible to prove (see \cite{DeAPe}) that if $\partial\cC$ is regular in the sense of diffusions for the interior of the stopping set, then $U\in C^1([0,T)\times\R^m)$. In that context, the bound on $L$ required by Assumption \ref{ass:1-2} is satisfied as soon as $\alpha^i$ and $r$ are continuous functions and $G\in C^{1,2}(\cD)$.
The requirement $G\in C^{1,2}(\cD)$ is actually rather mild and it allows to cover cases in which $G$ is not smooth everywhere in $[0,T]\times\R^m$. Perhaps the best known example of such behaviour is the case of American put/call options, where $G(x)=\max(K-x,0)$ and $G(x)=\max(x-K,0)$, respectively, for a constant $K>0$. In those examples $G$ is not even continuously differentiable in the whole space {\em but} it is smooth in the stopping set (see \cite{CDeAP} for a finite-horizon, multi-dimensional setting with $m=2$). The above discussion shows that in optimal stopping it is potentially rather easy to prove that Assumption \ref{ass:1-2} holds. Instead, it could be extremely difficult to obtain bounds on each of the second order derivatives $U_{x_i x_j}$, as required in other existing change of variable formulae. 
\begin{remark}[Continuous differentiability of $U$]
It may appear that the requirement $U\in C^1(\R_+\times\R^m)$ be much stronger than the usual smooth-fit condition in optimal stopping, which refers to continuity of {\em directional} derivatives. However, the smooth-fit condition is normally proved relying upon convergence of $\tau_*$ to zero in the limit as the initial point $\bm X_0=\bm x$ of the underlying process approaches $\partial\cC$ along a direction parallel to the $x_1$-axis (in the parametrisation of \eqref{eq:setC}). Such convergence is essentially equivalent to the concept of `regularity' of $\partial\cC$ in the sense of diffusions, which would also imply continuous differentiability of $U$ as shown in \cite{DeAPe}. 
\end{remark}

Another important aspect of our contribution concerns the assumptions made on the boundary. Monotonicity of the boundary as in Assumption \ref{ass:1-3} is often sufficient to prove regularity of $\partial\cC$ in the sense of diffusions (see, e.g., \cite{CDeAP}) and therefore continuous differentiability of the value function as needed in Assumption \ref{ass:1-2}. In optimal stopping for multi-dimensional Markov processes, it is often prohibitively difficult to prove geometric properties of the boundary $\partial\cC$ beyond the existence of a surface $b_1$ as in \eqref{eq:setC} and its monotonicity in each variable (Assumption \ref{ass:1-3}). Therefore, an application of other existing formulae is not normally possible as those require continuity (or higher regularity) of the boundary and/or, in some cases, properties of the process $b_1(t,X^2,\ldots,X^{m})$ (see, e.g., \cite{peskir2007change}). To give a sense of the challenge, let us consider for a moment $T=\infty$ and $m=2$, with $\bm X$ time-homogenous, $\Pi^t_s(\bm X)\equiv 0$ and $G$ independent of time. In this case \eqref{eq:setC}--\eqref{eq:setD} are given in terms of a function $x_2\mapsto b_1(x_2)$ and, in order to apply the formula from \cite{peskir2007change}, one needs to verify that $t\mapsto b_1(X^2_t)$ is a semi-martingale process. In the easiest (non-trivial) case of $X^2$ being a Brownian motion, the semimartingale property is essentially equivalent to convexity/concavity (or semi-convexity/concavity) of the mapping $x_2\mapsto b_1(x_2)$. That property is generally beyond reach in optimal stopping (a notable exception is \cite[Prop.\ 9]{ernst} for a specific problem). In contrast, monotonicity of optimal stopping boundaries (for $m\ge 2$) turns out to be more easily obtained in numerous examples as in \cite{CDeAP}, \cite{johnson2017quickest}, \cite{deangelis2017} and \cite{de2017dividend}, which will be discussed in detail in the next subsection. Thus, Assumption \ref{ass:1-3} is in line with properties that one may expect to be able to prove for optimal boundaries in tractable optimal stopping problems.

As for Assumption \ref{ass:1-1}, the requirement that coefficients $\beta^{ij}$ be locally Lipschitz is satisfied in all examples we are aware of in the literature. Since in our framework $\partial\cC$ is a set of zero Lebesgue measure in $\R_+\times\R^m$, then the condition $\P((t,\bm X_{t-})\in\partial\cC)=0$ for a.e.\ $t\ge 0$ is satisfied as soon as the law of $\bm X$ is absolutely continuous with respect to the Lebesgue measure.

\begin{remark}[\bf Running rewards]
When $\bm A\equiv 0$ the framework above can be immediately extended to include running rewards in the objective function. Consider a problem of the form
\begin{align}\label{eq:1}
\tilde U(t,\bm x)=\sup_{t\le \tau\le T}\E_{t,\bm x}\Big[\int_t^\tau \e^{-\Pi^t_s (\bm X)} R(s,\bm X_s)\d s+ \e^{-\Pi^t_\tau (\bm X)} \tilde G(\tau,\bm X_\tau)\Big],
\end{align}
where $R:\R_+\times\R^m\to\R$ is a continuous function. Assume we can find a function $F$ such that 
\begin{align}\label{eq:2}
\E_{t,\bm x}\Big[\e^{-\Pi^t_\tau (\bm X)} F(\tau,\bm X_\tau)\Big]=F(t,\bm x)-\E_{t,\bm x}\Big[\int_t^\tau \e^{-\Pi^t_s (\bm X)} R(s,\bm X_s)\d s\Big],
\end{align}
for every stopping time $\tau\in[t,T]$. Then, we can reduce \eqref{eq:1} to \eqref{eq:OS} by combining the two expressions above into 
\begin{align}\label{eq:3}
U(t,\bm x):=(\tilde U-F)(t,\bm x)=\sup_{t\le \tau\le T}\E_{t,\bm x}\Big[\e^{-\Pi^t_\tau (\bm X)} G(\tau,\bm X_\tau)\Big],
\end{align}
where $G:=\tilde G- F$. 

Finding the function $F$ in \eqref{eq:2}, in the case when $\Pi^t_s (\bm X)=\int_t^s r(u,\bm X_u)\d u$, boils down to solving a relatively simple PDE. Letting $\cL$ be the infinitesimal generator of $\bm X$, we need $F$ to be a solution of a Cauchy problem
\[
(\partial_t F+\cL F- r F)(t,\bm x)=-R(t,\bm x),\quad(t,\bm x)\in[0,T]\times\R^m,
\]
with sufficiently slow growth at infinity (e.g., linear growth is often sufficient). Notice that the terminal condition $F(T,\bm x)=\Phi(\bm x)$ can be chosen arbitrarily as it will affect the form of $U=\tilde U-F$ but it will not alter the equivalence between the problem in \eqref{eq:1} and the one in \eqref{eq:3}. In other words, changing $\Phi$ only changes the shift in $U=\tilde U-F$ and $G=\tilde G- F$. 

Instead, in the case of $\bm A\neq 0$, \eqref{eq:1} cannot be reduced to \eqref{eq:OS} and the analysis must be performed on a case-by-case basis. \end{remark}

\subsection{Four specific examples}

Optimal stopping problems on multi-dimensional underlying processes are appearing with increasing frequency in the literature (a probabilistic study of optimal stopping boundaries for multi-dimensional diffusions can be found for example in \cite{CCMS}). Here we briefly review four specific examples that fit within our framework. We emphasise that existing change of variable formulae cannot be applied to any of the examples below whereas our formula can be applied to all of them. \vspace{+5pt}

{\em American options with stochastic interest rate}. In \cite{CDeAP} we study the classical American put option problem under stochastic discounting and we apply directly results from this paper in the proofs of Propositions 3.9 and 3.12 therein. The problem in \cite{CDeAP} is set on a finite-time horizon and there is a two-dimensional continuous process $(X^1,X^2)=(R,X)$ where $R$ represents the discount rate and $X$ the stock price. An optimal boundary exists and can be parametrised either as a function of $(t,R)$ or as a function of $(t,X)$. Monotonicity of the boundary is shown as needed in Assumption \ref{ass:1-3} (\cite[Propositions\ 3.3 and 3.4]{CDeAP}). From that it also follows regularity of the boundary in the sense of diffusions and continuous differentiability of the value function (\cite[Thm.\ 3.5]{CDeAP}). Therefore Assumption \ref{ass:1-2} holds by virtue of the discussion around \eqref{eq:LLL} above. In terms of notation, in \eqref{eq:OS} we should take $G(t,r,x)=(K-x)^+$ and $\Pi^t_s(\bm X)=\int_t^s R_u\d u$. In particular we notice that $G$ is not smooth in the whole space, but it can be easily shown that $\cD\subset \{(t,r,x) : x\le K\}$ and therefore $G\in C^\infty(\cD)$. The application of our change of variable formula allows us to obtain both an integral equation for the optimal exercise boundary and, somewhat more importantly, a rigorous derivation of the Delta hedging portfolio. Although it is possible to prove that the value function is convex in $x$, no further information is available about its second order spatial derivatives. Moreover, it is not possible from the analysis in \cite{CDeAP} to conclude that either $t\mapsto b_1(t,X_t)$ or $t\mapsto b_2(t,R_t)$ is a semi-martingale. These difficulties prevent the use of alternative change of variable formulae.
\vspace{+3pt}

{\em Quickest detection}. In the context of quickest detection problems, multi-dimensional situations arise for example in \cite{johnson2017quickest}, \cite{gapeev} and \cite{ekstrom2020multi}. With particular reference to \cite{johnson2017quickest} by Johnson and Peskir, our change of variable formula could be directly applied in that setup. The problem is set on an infinite-time horizon and the underlying process is two-dimensional diffusive with the notation $(X^1,X^2)=(\Phi,X)$. It is shown in the paper (Corollary 8 and Proposition 16) that the optimal stopping boundary is a continuous monotonic mapping $x\mapsto b(x)$ (hence our Assumption \ref{ass:1-3} holds) and it is optimal to stop when $\Phi$ exceeds the moving boundary $b(X)$. The Mayer formulation of the problem (i.e., the analogue of our \eqref{eq:OS}) is provided in \cite[Proposition 3]{johnson2017quickest} and it corresponds to taking $\Pi^t_s(\bm X)=\lambda(s-t)$ and $G(t,\varphi,x)=-(1+\varphi)x^2-\alpha$ where $\lambda,\alpha>0$ are constants (notice that the minus sign is due to the fact that in \cite{johnson2017quickest} they address a minimisation problem). Johnson and Peskir prove in their Proposition 14 that their value function is continuously differentiable once in both variables. Therefore the argument that follows \eqref{eq:LLL} above implies that our Assumption \ref{ass:1-2} holds as well. 

We notice that Johnson and Peskir derive a nonlinear integral equation (\cite[Theorem 19]{johnson2017quickest}) by a non-trivial modification of the change of variable formula from \cite{peskir2005change}. Indeed they cannot directly apply results in \cite{peskir2005change} or \cite{peskir2007change} because they cannot prove that the process $t\mapsto b(X_t)$ (or suitable transformations thereof) is a semi-martingale. So they adopt an {\em ad-hoc} procedure. First, they perform a change of coordinate that leads to a two-dimensional process $(U,\Phi)$ with $U$ a process of bounded variation (notice that the notation is slightly clashing with our \eqref{eq:OS} but no confusion shall arise). After the change of coordinate, they parametrise the boundary of the continuation set with a function $\tilde b(u)$ which is the analogue of $b(x)$ in the original coordinate system. Second, they construct an approximation $U^n$ of the process $U$ in order to guarantee that the corresponding process $\tilde b(U^n)$ be also of bounded variation. Finally, they apply the change of variable formula from \cite{peskir2005change} to the process $(U^n,\Phi)$ with boundary $\tilde b(U^n)$ and show that it is then possible to pass to the limit as $n\to\infty$. Our Theorem \ref{thm:main} instead is directly applicable and thus the delicate construction in \cite{johnson2017quickest} could be avoided. 
\vspace{+5pt}

Many problems of stochastic singular control can be linked to optimal stopping and be solved via free boundary methods (see, e.g., \cite{deangelis2017}, \cite{bandini2020optimal}, \cite{Fe2018} for multi-dimensional diffusive set-ups). For those problems our change of variable formula is also useful.
\vspace{+3pt}

{\em Irreversible investment}. In \cite{deangelis2017} an irreversible investment problem with stochastic costs is connected to an optimal stopping problem on a two-dimensional non-degenerate diffusion $(X^1,X^2)=(Y,X)$. The problem in \cite{deangelis2017} is presented in its Bolza formulation but using our notation and transforming it into the Mayer formulation we have $G(t,x,y)=-y-f(x)$ and $\Pi^t_s(\bm X)=r(s-t)$, where $r>0$ is constant and
\[
f(x)=\E_x\Big[\int_0^\infty \e^{-r s}c(X_s)\d s\Big]
\]
for a suitable function $c$ specified in \cite{deangelis2017}. It is proven in \cite[Proposition 4.1]{deangelis2017} that it is optimal to stop when the process $Y$ falls below an optimal boundary $b(X)$. The mapping $x\mapsto b(x)$ is increasing and continuous (\cite[Proposition 4.4]{deangelis2017}), hence our Assumption \ref{ass:1-3} holds. The value function of the problem is continuously differentiable once in both variables $(x,y)$ (\cite[Proposition 4.3]{deangelis2017}) so that using the arguments around \eqref{eq:LLL} also our Assumption \ref{ass:1-2} is satisfied. 

Since the authors were unable to use any of the change of variable formulae known at the time, they resorted to a lengthy approximation procedure, based on variational inequalities, in order to arrive at an integral equation for the optimal boundary. In particular, the authors noticed that they could not apply \cite{peskir2007change} because they could not prove that the process $b(X_t)$ was a semi-martingale. Moreover, they had no bounds on the second order derivatives of the value function so that also alternative methods seemed to fail. Our Theorem \ref{thm:main} is directly applicable and could significantly simplify the derivation of the integral equation for the boundary.\vspace{+3pt}

{\em The dividend problem}. In \cite{bandini2020optimal} the dividend problem with stochastic discounting is connected to an optimal stopping problem on a two-dimensional reflecting diffusion $(X^1,X^2)=(Z,R)$. Here $Z$ is a reflecting Brownian motion with drift, upward-reflecting at a constant threshold $\alpha>0$, and $R$ is a CIR process. In the notation of \eqref{eq:OS} we have $G(t,z,r)=1$ and 
\[
\Pi^t_s(\bm X)=\int_t^s\rho(R_u)\d u-\lambda (L^\alpha_{s}-L^\alpha_{t}),
\] 
where $\lambda>0$ is a constant, $\rho$ is a suitable function and $L^\alpha$ is the local time of $Z$ at the reflection point $\alpha$. The two-dimensional dynamics with reflection is covered by our set-up in \eqref{eq:X-sde}. It is proven in \cite{bandini2020optimal} that it is optimal to stop when $Z$ exceeds a moving boundary $b(R)$ where $r\mapsto b(r)$ is continuous and decreasing (Lemma 3.8 and Theorem 3.13). Hence our Assumption \ref{ass:1-3} holds. Once again, the authors cannot control the second order derivatives of the value function near the optimal boundary and neither can they prove that the process $b(R_t)$ is a semi-martingale. However, it is shown in \cite[Proposition 3.11]{bandini2020optimal} that the value function is continuously differentiable once in $(z,r)$, so that our Assumption \ref{ass:1-2} holds. Then, our change of variable formula in Theorem \ref{thm:main} applies to the set-up in \cite{bandini2020optimal}.

\begin{remark}
Applications of our formula are natural also in the proof of so-called {\em verification theorems} in problems of stochastic singular control, including irreversible or partially reversible investment and the dividend problem (see \cite[Ch.\ VII.4]{fleming} and \cite[Ch. 4.5]{pham2009continuous} for an extensive coverage of the topic and financial motivations). 

In those problems, the controlled dynamics is inherently c\`adl\`ag due to the action of the singular control. In the context of our stochastic dynamics \eqref{eq:X-sde}, we can think of the process $\bm A$ as of a generic control. 

A verification theorem would normally require two ingredients:
\begin{itemize}
\item[(i)] $C^{1,2}$-regularity of a (candidate) value function $U$, obtained as a solution of a suitable HJB equation (or alternatively, $W^{1,2,p}$-regularity for sufficiently large $p>1$);
\item[(ii)] Existence of an optimal control. That is often obtained by constructing a Skorokhod reflection of $\bm X$ at a sufficiently regular free boundary, e.g., Lipschitz continuous as in \cite{SS91} or higher regularity as in \cite{SS89}.
\end{itemize} 
The proof of a verification theorem would then proceed via an application of It\^o's formula to $U(t,\bm X_t)$ and the use of the HJB equation. Thanks to our Theorem \ref{thm:main}, the smoothness requirements on the (candidate) value function are milder than those in $(i)$ above. As for the free boundary, our Assumption \ref{ass:1-3} is not overly restrictive if one considers that condition $(ii)$ is however required. 
\end{remark}

\subsection{Some sufficient conditions for Assumption \ref{ass:1-3}}
The study of optimal stopping boundaries has historically been developed mainly through examples due to the technical difficulties in formulating general theoretical results. In this spirit, the application of our Theorem \ref{thm:main} is subject to the verification of Assumptions \ref{ass:1-1}--\ref{ass:1-3}, which should be performed on a case-by-case basis. As explained in Section \ref{subs:framework}, the monotonicity of the optimal boundary (i.e., Assumption \ref{ass:1-3}) is often the key to the proof of continuous differentiability of the value function (as in the four examples above) and to the continuity of the optimal boundary (see, e.g., \cite{peskir2019continuity}, \cite{CDeAP2}). Moreover, it is also needed for the proof of our Theorem \ref{thm:main}. Therefore, it seems worth it to present some easily verifiable conditions that imply Assumption \ref{ass:1-3}. We emphasise though that such conditions are far from being necessary.

For the ease of presentation we focus on $\bm X\in\R^3$ but all results extend to higher dimension up to obvious changes. We are in the setting of \eqref{eq:OS} and $\bm X$ follows the dynamics given in \eqref{eq:X-sde}. Throughout the section we assume: 
\begin{assumption}\label{ass:stand}
The following conditions are satisfied:
\begin{itemize}
\item [(i)] The process $\bm A$ is non-decreasing and \eqref{eq:X-sde} admits a unique strong solution. 
\item [(ii)] The discount rate $r(t,\bm X_{t-})=r\ge 0$ is constant in \eqref{eq:lambda}. 
\item[(iii)]For $i,j=1,2,3$ it holds $\sigma^{ij}(t,\bm x)=\sigma^{ij}(x_i)$ and $\gamma^i(t,\bm x)=\gamma^i=const$.
Moreover, there is a function $h:\R_+\to\R_+$ such that $\int_0^\eps h^{-2}(y)\d y=\infty$ for all $\eps>0$ and for $j=1,2,3$
\[
\big|\sigma^{1j}(x_1)-\sigma^{1j}(x_1')\big|\le h(|x_1-x_1'|).
\] 
\item [(iv)] For $i=2,3$ it holds $\alpha^i(t,\bm x)=\alpha^i(x_i)$ whereas $\alpha^1(t,\bm x)=\alpha^1(\bm x)$
and $z\mapsto \alpha^1(z,x_2,x_3)$ is Lipschitz continuous uniformly in $(x_2,x_3)$.
\item[(v)] The stopping time $\tau_*$ in \eqref{eq:taustar} is optimal and the continuation set is specified by \eqref{eq:setC}.
\end{itemize}
\end{assumption}
The process is time-homogeneous, the diffusion coefficient of the $i$-th coordinate depends at most on $X^i$ and the coefficient in front of $\bm A$ is constant. The coefficients in the dynamics of $X^2$ and $X^3$ only depend on $X^2$ and $X^3$, respectively, whereas $\alpha^1$ depends on the dynamics of $(X^1,X^2,X^3)$ and it is Lipschitz in its first variable. These assumptions, combined with monotonicity of $\bm A$, will enable the use of a simple variation on standard comparison principles for the trajectories of the process $\bm X$ starting from different initial conditions (see Appendix). Taking a constant discount rate simplifies the discussion below but it is shown in the setting of  \cite{bandini2020optimal} and \cite{CDeAP} that it is possible to consider stochastic discount rates and obtain monotonic optimal boundaries.

We consider two specific cases. In the first case the gain function is time homogeneous and it only depends explicitly on $X^1$, while it depends implicitly on $(X^2,X^3)$ via the drift of $X^1$. This situation is similar to the American put option problem studied in \cite{CDeAP} and we do not assume smoothness of $G$. An interpretation of this problem is that $G$ is the payoff of an American option written on an underlying asset with value $X^1$. The rate of return of such an asset (the drift of $X^1$) depends on the values of two other stochastic factors, $X^2$ and $X^3$. As observed above the drifts of $X^2$ and $X^3$ are decoupled, in the sense of $(iv)$ in Assumption \ref{ass:stand}. Notice however that $X^1$, $X^2$ and $X^3$ are also correlated by the Brownian motions $W^1$, $W^2$ and $W^3$ via the coefficients $\sigma^{ij}$. We state the next result for $G$ non-decreasing but it will be clear that an analogous result holds for non-increasing $G$, up to trivial changes. 
\begin{proposition}\label{prop:monot1}
Let Assumption \ref{ass:stand} hold. If $G(t,\bm x)=G(x_1)$, then $t\mapsto b_1(t,x_2,x_3)$ is non-decreasing. 
If $z\mapsto G(z)$ is non-decreasing, $z\mapsto \alpha^1(x_1,z,x_3)$ is non-decreasing and $z\mapsto \alpha^1(x_1,x_2,z)$ is non-increasing, then $x_2\mapsto b_1(t,x_2,x_3)$ is non-increasing and $x_3\mapsto b_1(t,x_2,x_3)$ is non-decreasing.
\end{proposition}
\begin{proof}
Since $G$ is independent of time and $\bm X$ is time-homogeneous, then 
\begin{align}\label{eq:Uhomo}
U(t,\bm x)=\sup_{t\le \tau\le T}\E_{t,\bm x}\Big[\e^{-r(\tau-t)} G(\bm X_\tau)\Big]=\sup_{0\le \tau\le T-t}\E_{0,\bm x}\Big[\e^{-r\tau} G(\bm X_\tau)\Big],
\end{align}
because the law of $(\bm X_s)_{s\ge t}$ under $\P_{t,\bm x}$ is the same as the law of $(\bm X_s)_{s\ge 0}$ under $\P_{0,\bm x}$ and the class of the stopping times is adjusted accordingly. From the final expression above it is not hard to verify that $t\mapsto U(t,\bm x)$ is non-increasing. In particular, that implies
\[
(t,\bm x)\in\cD\implies (s,\bm x)\in \cD,\quad \text{for all $s\in[t,T]$},
\] 
from which we conclude that $t\mapsto b_1(t,x_2,x_3)$ must be non-decreasing.

Next we want to prove that for any $\eps>0$
\begin{align}\label{eq:Umon}
U(t,x_1,x_2+\eps,x_3)\ge U(t,x_1,x_2,x_3)\quad\text{and}\quad U(t,x_1,x_2,x_3+\eps)\le U(t,x_1,x_2,x_3).
\end{align}
Indeed, if \eqref{eq:Umon} holds then, thanks to the fact that $G(t,\bm x)=G(x_1)$, we have 
\[
U(t,x_1,x_2+\eps,x_3)-G(x_1)\ge U(t,x_1,x_2,x_3)-G(x_1)
\]
and   
\[
U(t,x_1,x_2,x_3+\eps)-G(x_1)\le U(t,x_1,x_2,x_3)-G(x_1),
\]
which imply $x_2\mapsto b_1(t,x_2,x_3)$ non-increasing and $x_3\mapsto b_1(t,x_2,x_3)$ non-decreasing, respectively.

In order to prove the first inequality in \eqref{eq:Umon}, let us set ${\bm x}_\eps=(x_1,x_2+\eps,x_3)$ and denote by $X^{i;{\bm x}_\eps}$ the $i$-th coordinate of the process $\bm X$ started from $\bm X_0={\bm x}_\eps$. Analogously we use $X^{i;{\bm x}}$ for the $i$-th coordinate of the process $\bm X$ started from $\bm X_0={\bm x}$. Thanks to $(iii)$ and $(iv)$ in Assumption \ref{ass:stand}, it is clear that $X^{3;\bm x_\eps}_s=X^{3;\bm x}_s$ and $X^{2;\bm x_\eps}_s\ge X^{2;\bm x}_s$ for all $s\ge 0$, $\P$-a.s.\ by pathwise uniqueness. It then follows that 
\[
\alpha^1(\,\cdot\,,X^{2;\bm x_\eps}_s,X^{3;\bm x_\eps}_s)\ge \alpha^1(\,\cdot\,,X^{2;\bm x}_s,X^{3;\bm x}_s).
\]
The comparison principle in Proposition \ref{prop:comparison} in Appendix applies with $C_t=\gamma^1 A^1_t$, 
\[
\eta^1(\omega,s,y)= \alpha^1(y,X^{2;\bm x}_s(\omega),X^{3;\bm x}_s(\omega))\quad\text{and}\quad\eta^2(\omega,s,y)= \alpha^1(y,X^{2;\bm x_\eps}_s(\omega),X^{3;\bm x_\eps}_s(\omega)).
\]
Then $X^{1;\bm x_\eps}_s\ge X^{1;\bm x}_s$ for all $s\ge 0$, $\P$-a.s. The assumed monotonicity of $G$ then implies $G(X^{1;\bm x_\eps}_\tau)\ge G(X^{1;\bm x}_\tau)$, which translates into
\[
U(t,\bm x_\eps)=\sup_{0\le \tau\le T-t}\E_{0,\bm x_\eps}\Big[\e^{-r\tau} G(\bm X_\tau)\Big]\ge\sup_{0\le \tau\le T-t}\E_{0,\bm x}\Big[\e^{-r\tau} G(\bm X_\tau)\Big]= U(t,\bm x)
\]
as needed. The second inequality in \eqref{eq:Umon} is proven by analogous arguments and we omit further details for brevity.
\end{proof}

In the second case, we remove the assumption of time-homogeneity of the payoff and we allow it to depend on all three coordinates of $\bm X$. However, we require an additional assumption on the smoothness of $G$. The next proposition extends to our (multi-dimensional) setting ideas contained in \cite{sd1992finite} for the case of $\bm X\in\R$ and without the bounded variation term $\bm A$ in the dynamics. We introduce some necessary notation: for $G\in C^{1,2}([0,T]\times\R^3)$ we let 
\[
H(t,\bm x):=\Big(G_t+\tfrac{1}{2}\sum_{i,j=1}^3\beta^{ij}G_{x_i x_j}+\sum_{i=1}^3\alpha^i G_{x_i}-rG\Big)(t,\bm x)
\]
and we will always assume that Dynkin's formula holds, i.e., for any stopping time $\tau\in[0,T-t]$
\begin{align}\label{eq:hL2}
&\E_{0,\bm x}\Big[\e^{-r\tau} G(t+\tau,\bm X_\tau)\Big]=G(t,\bm x)\\
&\qquad+\!\E_{0,\bm x}\Big[\!\int_0^{\tau}\!\!\e^{-r s}H(t\!+\!s,\bm X_s)\d s\!+\!\sum_{i=1}^3\!\int_0^\tau\!\!\e^{-r s}\gamma^i G_{x_i}(t\!+\!s,\bm X_{s-})\d A^{c,i}_s\Big]\notag\\
&\qquad+\!\E_{0,\bm x}\Big[\sum_{s\le \tau}\big(G(t\!+\!s,\bm X_{s})-G(t\!+\!s,\bm X_{s-})\big)\Big].\notag
\end{align}
The latter is guaranteed if, for example, $G_{x_i}$ is bounded for $i=1,2,3$. 

The aim of the next proposition is to illustrate a method to check monotonicity of the boundary in practical problems. We do not claim to provide the most general result possible because that is outside the scope of this section.
\begin{proposition}
Let Assumption \ref{ass:stand} hold and assume $G\in C^{1,2}(\R_+\times\R^3)$ and \eqref{eq:hL2}. Let us further assume $\gamma^i\ge 0$ for $i=1,2,3$.
\begin{itemize}
\item[(i)] If $\bm A\equiv 0$ and $t\mapsto H(t,\bm x)$ is non-increasing, then $t\mapsto b_1(t,x_2,x_3)$ is non-decreasing.
\item[(ii)] If $\bm A\neq 0$ and $t\mapsto H(t,\bm x)$, $t\mapsto G_{x_i}(t,\bm x)$ are non-increasing for $i=1,2,3$, then $t\mapsto b_1(t,x_2,x_3)$ is non-decreasing.
\item[(iii)] If $\bm A\equiv 0$ and 
\begin{itemize}
\item[(iii.a)] $z\mapsto \alpha^1(x_1,z,x_3)$ is non-decreasing and $z\mapsto \alpha^1(x_1,x_2,z)$ is non-increasing; 
\item[(iii.b)] $z\mapsto H(z,x_2,x_3)$, $z\mapsto H(x_1,z,x_3)$ are non-decreasing and $z\mapsto H(x_1,x_2,z)$ is non-increasing;
\end{itemize}
Then $z\mapsto b_1(t,z,x_3)$ is non-increasing and $z\mapsto b_1(t,x_2,z)$ is non-decreasing.
\item[(iv)] If $\bm A\neq 0$, conditions $(iii.a)$ and $(iii.b)$ hold and for $i=1,2,3$
\begin{itemize}
\item[(iv.a)] $z\mapsto G_{x_i}(z,x_2,x_3)$, $z\mapsto G_{x_i}(x_1,z,x_3)$ are non-decreasing and $z\mapsto G_{x_i}(x_1,x_2,z)$ is non-increasing;
\end{itemize}
Then $z\mapsto b_1(t,z,x_3)$ is non-increasing and $z\mapsto b_1(t,x_2,z)$ is non-decreasing.
\end{itemize}
\end{proposition}
\begin{proof}
We start by proving $(ii)$. The proof of $(i)$ is analogous but easier, thus it is omitted. To prove $(ii)$ it suffices to show that for any $\eps\in(0,t)$ 
\begin{align}\label{eq:mon-t} 
U(t,\bm x)-G(t,\bm x)\le U(t-\eps,\bm x)-G(t-\eps,\bm x).
\end{align}
To keep a compact notation we set $t_\eps=t-\eps$. Thanks to \eqref{eq:hL2} 
\begin{align*}
&U(t,\bm x)-G(t,\bm x)\\
&=\sup_{0\le\tau\le T-t}\E_{0,\bm x}\Big[\!\int_0^{\tau}\!\!\e^{-r s}H(t\!+\!s,\bm X_s)\d s\!+\!\sum_{i=1}^3\!\int_0^\tau\!\!\e^{-r s}\gamma^i G_{x_i}(t\!+\!s,\bm X_{s-})\d A^{c,i}_s\notag\\
&\qquad\qquad\qquad\qquad+\sum_{s\le \tau}\big(G(t\!+\!s,\bm X_{s})-G(t\!+\!s,\bm X_{s-})\big)\Big]\\
&\le \sup_{0\le\tau\le T-t_\eps}\E_{0,\bm x}\Big[\!\int_0^{\tau}\!\!\e^{-r s}H(t_\eps\!+\!s,\bm X_s)\d s\!+\!\sum_{i=1}^3\!\int_0^\tau\!\!\e^{-r s}\gamma^i G_{x_i}(t_\eps\!+\!s,\bm X_{s-})\d A^{c,i}_s\notag\\
&\qquad\qquad\qquad\qquad+\sum_{s\le \tau}\big(G(t_\eps\!+\!s,\bm X_{s})-G(t_\eps\!+\!s,\bm X_{s-})\big)\Big]\\
&=U(t_\eps,\bm x)-G(t_\eps,\bm x),
\end{align*}
where the inequality uses the monotonicity of $H$ and $G_{x_i}$ and the fact that $[0,T-t]\subset[0,T-t_\eps]$, i.e., the class of admissible stopping times for $U(t_\eps,\bm x)$ is larger than for $U(t,\bm x)$. It is worth noticing that for the jump terms we are indeed using 
\begin{align}\label{eq:gradG}
&G(t\!+\!s,\bm X_{s})-G(t\!+\!s,\bm X_{s-})=\int_0^1\langle\nabla G(t\!+\!s,\bm X_{s-}\!+\!u\,\bm\gamma\Delta\bm A_s),\bm\gamma\Delta\bm A_s\rangle\d u\\
&\le\int_0^1\langle\nabla G(t_\eps\!+\!s,\bm X_{s-}\!+\!u\,\bm\gamma\Delta\bm A_s),\bm\gamma\Delta\bm A_s\rangle\d u= G(t_\eps\!+\!s,\bm X_{s})-G(t_\eps\!+\!s,\bm X_{s-})\notag
\end{align}
where $\langle\cdot ,\cdot\rangle$ is the scalar product, $\bm\gamma\Delta\bm A:=(\gamma^1\Delta A^1,\gamma^2\Delta A^2,\gamma^3\Delta A^3)$ and the inequality holds because $\gamma^i\Delta A^i_s\ge 0$ and $G_{x_i}$ is monotonic in time.

Now we prove $(iv)$. The proof of $(iii)$ is analogous but easier and thus it is omitted. It suffices to prove that for any $\eps>0$
\begin{equation}\label{eq:Umon2}
\begin{split}
&(U-G)(t,x_1,x_2+\eps,x_3)\ge (U-G)(t,x_1,x_2,x_3)\quad\text{and}\\
&(U-G)(t,x_1,x_2,x_3+\eps)\le (U-G)(t,x_1,x_2,x_3).
\end{split}
\end{equation}
In order to prove the first inequality in \eqref{eq:Umon2}, let us set ${\bm x}_\eps=(x_1,x_2+\eps,x_3)$ and denote by $X^{i;{\bm x}_\eps}$ the $i$-th coordinate of the process $\bm X$ started from $\bm X_0={\bm x}_\eps$. Analogously we use $X^{i;{\bm x}}$ for the $i$-th coordinate of the process $\bm X$ started from $\bm X_0={\bm x}$. Thanks to $(iii)$ and $(iv)$ in Assumption \ref{ass:stand}, it is clear that $X^{3;\bm x_\eps}_s=X^{3;\bm x}_s$ and $X^{2;\bm x_\eps}_s\ge X^{2;\bm x}_s$ for all $s\ge 0$, $\P$-a.s.\ by pathwise uniqueness. It then follows that 
\[
\alpha^1(\,\cdot\,,X^{2;\bm x_\eps}_s,X^{3;\bm x_\eps}_s)\ge \alpha^1(\,\cdot\,,X^{2;\bm x}_s,X^{3;\bm x}_s).
\]
Analogously to the proof of Proposition \ref{prop:monot1} we apply the comparison principle in Proposition \ref{prop:comparison} in Appendix and obtain $X^{1;\bm x_\eps}_s\ge X^{1;\bm x}_s$ for all $s\ge 0$, $\P$-a.s. It follows by the monotonicity of $H$ and $G_{x_i}$ that
\[
H(t+s,\bm X^{\bm x_\eps}_s)\ge H(t+s,\bm X^{\bm x}_s)\quad\text{and}\quad G_{x_i}(t+s,\bm X^{\bm x_\eps})\ge G_{x_i}(t+s,\bm X^{\bm x}), 
\]
for $i=1,2,3$. By the inequalities above we obtain
\begin{align*}
&U(t,\bm x_\eps)-G(t,\bm x_\eps)\\
&=\sup_{0\le\tau\le T-t}\E_{0,\bm x_\eps}\Big[\!\int_0^{\tau}\!\!\e^{-r s}H(t\!+\!s,\bm X_s)\d s\!+\!\sum_{i=1}^3\!\int_0^\tau\!\!\e^{-r s}\gamma^i G_{x_i}(t\!+\!s,\bm X_{s-})\d A^{c,i}_s\notag\\
&\qquad\qquad\qquad\qquad+\sum_{s\le \tau}\int_0^1\langle\nabla G(t\!+\!s,\bm X_{s-}\!+\!u\,\bm\gamma\Delta\bm A_s),\bm\gamma\Delta\bm A_s\rangle\d u\Big]\\
&\ge \sup_{0\le\tau\le T-t}\E_{0,\bm x}\Big[\!\int_0^{\tau}\!\!\e^{-r s}H(t\!+\!s,\bm X_s)\d s\!+\!\sum_{i=1}^3\!\int_0^\tau\!\!\e^{-r s}\gamma^i G_{x_i}(t\!+\!s,\bm X_{s-})\d A^{c,i}_s\notag\\
&\qquad\qquad\qquad\qquad+\sum_{s\le \tau}\int_0^1\langle\nabla G(t\!+\!s,\bm X_{s-}\!+\!u\,\bm\gamma\Delta\bm A_s),\bm\gamma\Delta\bm A_s\rangle\d u\Big]\\
&=U(t,\bm x)-G(t,\bm x),
\end{align*}
as needed. The second inequality in \eqref{eq:Umon2} can be proven by analogous arguments up to obvious changes.
\end{proof}

\section{Proof of Theorem \ref{thm:main}}\label{sec:proof}

We first prove our result in Section \ref{sec:easy}, in the case when 
\begin{align}\label{eq:monot}
\text{$b_1$ is non-decreasing in $t$ and in $x_i$, for $i=2,\ldots,m$.}
\end{align} 
The remaining cases in Assumption \ref{ass:1-3} will be discussed later, in Section \ref{sec:med}, as they only require minor changes to the arguments of proof. 

\subsection{Proof under \eqref{eq:monot}}\label{sec:easy}
We regularise our function $U$ to obtain an approximating sequence 
\[
(U^n)_{n\ge 1}\subset C^{1,2}(\R_+\times\R^m)
\] 
defined by
\begin{equation}
\label{eq:u_n}
\begin{aligned}
U^{n}(t,\bm{x}):&=n^m\int_{x_1}^{x_1+1/n}...\int_{x_m}^{x_m+1/n} U(t,z_1,\ldots,z_m)\d z_1\ldots\d z_m\\
&=n^m\int_{\Lambda_{n}(\bm{x})}U(t,\bm{z})\d\bm{z},
\end{aligned}
\end{equation}
where $\Lambda_n(\bm x):=\bigtimes_{k=1}^{m}[x_k,x_k+1/n]$. 
Since $U \in C^1(\R_+\times\R^m)$, then it is clear that $U^n \in C^{1,2}(\R_+\times\R^m)$. Its derivatives read
\begin{align}
&U^n_{t}(t,\bm x)=n^m\int_{\Lambda_n(\bm x)}U_t(t,\bm{z})\d \bm{z},\label{eq:un_t}\\
&U^n_{x_i}(t,\bm x)=n^m\int_{\Lambda_n(\bm x)}U_{x_i}(t,\bm{z})\d \bm{z},\label{eq:un_i}
\end{align}
\begin{align}
&U^n_{x_i x_j}(t,\bm x)=n^m\int_{\Lambda^{-i}_n(\bm x)}\big[U_{x_j}(t,x_i+1/n,\bm{z}_{-i})-U_{x_j}(t,x_i,\bm{z}_{-i})\big]\d \bm{z}_{-i}\label{eq:un_ij}\\ 
&\qquad\qquad\:\:=n^m\int_{\Lambda^{-j}_n(\bm x)}\big[U_{x_i}(t,x_j+1/n,\bm{z}_{-j})-U_{x_i}(t,x_j,\bm{z}_{-j})\big]\d \bm{z}_{-j},\nonumber
\end{align} 
for any $i,j \in 1,\ldots,m$, where we use the notations 
\begin{equation}\label{notations}
\begin{split}
&\Lambda^{-i}_n(\bm x):=\big(\bigtimes_{k=1}^{i-1}[x_k,x_k+1/n]\big)\times \big(\bigtimes_{k=i+1}^{m}\![x_k,x_k+1/n]\big),\\
&\bm{z}_{-i}:=(z_1,\ldots z_{i-1},z_{i+1},\ldots z_m),
\end{split}
\end{equation}
and the second equality in \eqref{eq:un_ij} is simply by integration.
Although $U_{x_i x_j}$ fails to be continuous at the boundary $\partial\cC$, for each $(t,\bm x) \notin \partial\cC$ there is a large enough $n$ such that
\begin{align*}
&U^n_{x_i x_j}(t,\bm x)=n^m\int_{\Lambda_n(\bm x)}U_{x_i x_j}(t,\bm{z})\d\bm{z}.
\end{align*} 
Consequently, for $i,j=1,...,m$, and for any compact $K\subset \R_+\times\R^m$ we have 
\begin{equation}
\label{eq:ptwconv-1}
\begin{aligned}
&\lim_{n\uparrow \infty} \sup_{(t,\bm x)\in K}\Big(\big|U^n-U\big|(t,\bm x)+\big|U^n_{t}-U_{t}\big|(t,\bm x)+\sum_{i=1}^m\big|U^n_{x_i}-U_{x_i}\big|(t,\bm x)\Big)=0,\\
&\lim_{n\uparrow \infty} U^n_{x_i x_j}(t,\bm x)=U_{x_i x_j}(t,\bm x), \quad \text{for all $(t,\bm x) \in \big(\R_+\times\R^m\big)\setminus \partial \cC$}. 
\end{aligned}
\end{equation}

For $\delta>0$, let us set 
\begin{equation}
\label{eq:V-delta}
V^{\delta}:=[0,1/\delta]\times[-1/\delta,1/\delta]^m,
\end{equation}
and
\begin{equation}
\label{eq:tau_de}
\tau_\delta:=\inf\{t\ge 0: (t,\bm{X}_t)\notin V^{\delta}\}.
\end{equation}

Applying It\^o's formula to $U^n(t\wedge\tau_{\delta},\bm X_{t\wedge\tau_{\delta}})$, we obtain 
\begin{equation}
\label{eq:ito-un}
\begin{aligned}
U^n&(t\wedge\tau_{\delta},\bm{X}_{t\wedge\tau_{\delta}})=U^n(0,\bm{x})\\
&+\int_0^{t\wedge\tau_{\delta}} \!\Big[\Big(U^n_t+\sum_{i=1}^m \alpha^i U^n_{x_i}\Big)(u,\bm{X}_{u-})+\tfrac{1}{2}\sum_{i,j=1}^m\mathds{1}_{\{(u,\bm{X}_{u-})\notin\partial\cC\}}\big(\beta^{ij}U^n_{x_i x_j}\big)(u,\bm{X}_{u-})\Big]\d u\\
&+\sum_{i=1}^m\int_0^{t\wedge\tau_{\delta}}\! \big(\gamma^i U^n_{x_i}\big)(u,\bm{X}_{u-})\d A^{i}_u\\
&+\!\sum_{u\le t\wedge\tau_{\delta}}\!\!\Big(U^n(u,\bm X_u)\!-\!U^n(u,\bm X_{u-})\!-\!\sum_{i=1}^m \big(\gamma^i U^n_{x_i}\big)(u,\bm X_{u-})\Delta A^{i}_u\Big)\\
&+\sum_{i,j=1}^{m}\int_{0}^{t\wedge\tau_{\delta}}U^n_{x_i}(u,\bm{X}_{u-})\sigma^{ij}(u,\bm{X}_{u-}) \d B^{j}_u,\qquad\text{for $t\in [0,\infty)$, $\P$-a.s.}
\end{aligned}
\end{equation}
having also used $\P((t,\bm{X}_{t-})\in\partial \cC)=0$ for a.e.\ $t\ge 0$ by Assumption \ref{ass:1-1}. Since the jumps of the process $\bm X$ only arise from the bounded variation process $\bm A$, we can also simplify the expression above by writing
\begin{align}\label{eq:jj}
&\sum_{i=1}^m\int_0^{t\wedge\tau_{\delta}} \big(\gamma^i U^n_{x_i}\big)(u,\bm{X}_{u-})\d A^{i}_u\notag\\
&+\sum_{u\le t\wedge\tau_{\delta}}\Big(U^n(u,\bm X_u)-U^n(u,\bm X_{u-})-\sum_{i=1}^m \big(\gamma^i U^n_{x_i}\big)(u,\bm X_{u-})\Delta A^{i}_u\Big)\\
&=\sum_{i=1}^m\int_0^{t\wedge\tau_{\delta}} \big(\gamma^i U^n_{x_i}\big)(u,\bm{X}_{u-})\d A^{c,i}_u+\sum_{u\le t\wedge\tau_{\delta}}\Big(U^n(u,\bm X_u)-U^n(u,\bm X_{u-})\Big),\notag
\end{align}  
by using the decomposition $A^{i}_t=A^{c,i}_t+\sum_{s\le t}\Delta A^i_s$ with $A^{c,i}$ the continuous part of the process $A^i$. Letting $n\to\infty$ (possibly along a subsequence) all terms involving only $U^n$ and its first derivatives (including the stochastic integral and the jump terms) converge to their analogue for the function $U$, thanks to the uniform convergence in \eqref{eq:ptwconv-1} and the fact that $(u,\bm X_{u-})\in V^\delta$ for $u\in[0,t\wedge\tau_{\delta}]$. 

It is worth looking in some detail at the jump terms. First we split the sum of jumps as
\begin{align*}
&\sum_{u\le t\wedge\tau_{\delta}}\Big(U^n(u,\bm X_u)-U^n(u,\bm X_{u-})\Big)\\
&=\sum_{u< t\wedge\tau_{\delta}}\Big(U^n(u,\bm X_u)-U^n(u,\bm X_{u-})\Big)+\Big(U^n(t\wedge\tau_{\delta},\bm X_{t\wedge\tau_{\delta}})-U^n(t\wedge\tau_{\delta},\bm X_{t\wedge\tau_{\delta}-})\Big).
\end{align*}
For the final term, letting $n\to\infty$ we can use pointwise convergence to get 
\[
\lim_{n\to\infty}\Big(U^n(t\wedge\tau_{\delta},\bm X_{t\wedge\tau_{\delta}})-U^n(t\wedge\tau_{\delta},\bm X_{t\wedge\tau_{\delta}-})\Big)=\Big(U(t\wedge\tau_{\delta},\bm X_{t\wedge\tau_{\delta}})-U(t\wedge\tau_{\delta},\bm X_{t\wedge\tau_{\delta}-})\Big).
\]
For the remaining sum we have 
\begin{align*}
&\sum_{u< t\wedge\tau_{\delta}}\Big(U^n(u,\bm X_u)-U^n(u,\bm X_{u-})\Big)\\
&=\sum_{u< t\wedge\tau_{\delta}}\int_0^1\langle\nabla U^n(u,\bm X_{u-}+\lambda\bm\gamma(u,\bm X_{u-})\Delta \bm A_u),\bm\gamma(u,\bm X_{u-})\Delta \bm A_u\rangle \d \lambda,
\end{align*}
where $\bm\gamma(u,\bm X_{u-})\Delta \bm A_u$ is the vector with entries $\gamma^i(u,\bm X_{u-})\Delta A^i_u$. We claim that the dominated convergence theorem holds and therefore it gives
\begin{equation}\label{eq:dom}
\begin{split}
&\lim_{n\to\infty}\sum_{u< t\wedge\tau_{\delta}}\int_0^1\langle\nabla U^n(u,\bm X_{u-}+\lambda\bm\gamma(u,\bm X_{u-})\Delta \bm A_u),\bm\gamma(u,\bm X_{u-})\Delta \bm A_u\rangle \d \lambda\\
&=\sum_{u< t\wedge\tau_{\delta}}\int_0^1\langle\nabla U(u,\bm X_{u-}+\lambda\bm\gamma(u,\bm X_{u-})\Delta \bm A_u),\bm\gamma(u,\bm X_{u-})\Delta \bm A_u\rangle \d \lambda,
\end{split}
\end{equation}
which concludes convergence of the jump terms in \eqref{eq:ito-un} and \eqref{eq:jj}.
It remains to justify the use of the dominated convergence theorem in \eqref{eq:dom}. 

For $(t,\bm x)\in V^\delta$ and sufficiently large $n$ it holds $\|\nabla U^n(t,\bm x)\|\le 1+\|\nabla U(t,\bm x)\|\le c_\delta$ and $|\gamma^i(t,\bm x)|\le c_\delta$ for suitable $c_\delta>0$, with $\|\cdot\|$ the Euclidean norm. Since $(u,\bm X_{u})\in V^\delta$ for $u\in[0,t\wedge\tau_{\delta})$, then
\[
\big|\langle\nabla U^n(u,\bm X_{u-}+\lambda\bm\gamma(u,\bm X_{u-})\Delta \bm A_u),\bm\gamma(u,\bm X_{u-})\Delta \bm A_u\rangle\big|\le c_\delta^2 \sum_{i=1}^m|\Delta A^i_u|
\]
and, moreover, 
\[
\sum_{u< t\wedge\tau_{\delta}}\sum_{i=1}^m|\Delta A^i_u|\le \sum_{i=1}^m|A^i|_{t\wedge\tau_{\delta}}<\infty,\quad\text{$\P$-a.s.},
\]
where we recall that $|A^i|_{t\wedge\tau_{\delta}}$ is the total variation of $A^i$ on $[0,t\wedge\tau_{\delta}]$. This justifies the use of dominated convergence theorem in \eqref{eq:dom}.

We now turn our attention to the terms in \eqref{eq:ito-un} that involve the second order spatial derivatives. If we can justify the use of dominated convergence to pass limits under the integral for those terms, then using the second limit in \eqref{eq:ptwconv-1} we obtain \eqref{eq:Ito-0}, upon also letting $\delta\downarrow 0$ at the end.

Since $U$ is twice continuously differentiable in space at all points off the boundary $\partial\cC$ and given that $\P((t,\bm{X}_{t-})\in\partial\cC)=0$ for a.e.\ $t\ge 0$, then it is enough to prove that there exists a constant $C_\delta>0$ independent of $n$, such that
\begin{equation}
\label{eq:goal}
\sup_{(t,\bm x) \in V^{\delta}}\left|\sum_{i,j=1}^m\beta^{ij}(t,\bm x)U^n_{x_i x_j}(t,\bm x)\right|\leq C_{\delta}.
\end{equation}
We accomplish our task in two steps.

{\em Step 1}. We show that for any $(t,\bm x)\in  V^{\delta}$ and $n$ fixed, $U^n_{x_i x_j}(t,\bm x)$ admits the representation: 
\begin{equation}
\begin{aligned}
\label{eq:goal-1}
U^n_{x_i x_j}(t,\bm x)=&n^m\int_{\Lambda_n(\bm x)}U_{x_i x_j}(t,\bm{z})\mathds{1}_{\{z_1\geq b^{\eps}_1(t,z_2,...,z_m)\}}\d\bm{z}\\
&+n^m\int_{\Lambda_n(\bm x)}U_{x_i x_j}(t,\bm{z})\mathds{1}_{\{z_1\le b_1(t,z_2,...,z_m)\}}\d\bm{z}
+F^{n,\eps}_{ij}(t,\bm x), \quad \forall \eps>0,
\end{aligned}
\end{equation}
for any $i,j=1,...,m$, where $F^{n,\eps}_{ij}$ is a remainder that we will show converges to zero and $b^{\eps}_1:\R_+\times\R^{m-1}\to \R$ is defined as
\begin{equation}
\label{eq:b-1-eps}
b^{\eps}_1(t,z_2,\ldots z_m):=b_1(t+\eps,z_2+\eps,z_3+\eps,\ldots z_{m}+\eps)+\eps.
\end{equation}
Recall the compact notation $\bm{z}_{-i}$ from \eqref{notations}. Since we are currently assuming that $b_1$ is non-decreasing in all variables, the limit:
\[
b^{0+}_1(t,\bm{z}_{-1}):=\lim_{\eps\downarrow 0}b^{\eps}_1(t,\bm{z}_{-1}),
\]
exists and $b^{0+}_1(t,\bm{z}_{-1})\ge b_1(t,\bm{z}_{-1})$. Using that $\cD$ is closed then  
\[
\cD\ni\big(t+\eps, b^\eps_1(t,\bm{z}_{-1})-\eps,z_2+\eps,\ldots z_m+\eps\big)\to \big(t,b^{0+}_1(t,\bm{z}_{-1}),z_2,\ldots z_m\big)\in\cD,
\]
as $\eps\downarrow 0$ and, therefore, $b^{0+}_1(t,\bm{z}_{-1})\le b_1(t,\bm{z}_{-1})\le b^{0+}_1(t,\bm{z}_{-1})$ by definition of the set $\cD$. The reason for introducing the function $b^\eps_1$ is that the set
\begin{align}\label{eq:C1e}
\cC^\eps_1:=\{(t,\bm x)\in\R_+\times\R^m : x_1>b^\eps_{1}(t,\bm x_{-1})\}
\end{align}
is such that its closure is strictly contained in $\cC$ for all $\eps>0$, i.e., 
\begin{align}\label{eq:C1eC}
\overline {\cC^\eps_1}\subset \cC.
\end{align}
The latter fact will be used several times, along with the fact that $U_{x_i x_j}\in C(\overline{\cC^\eps_1})$.

Let us start with $i=1$ (or $j=1$) and using the expression in \eqref{eq:un_ij}, let us re-write the integral by considering separately the cases in which the interval $[x_1,x_1+1/n]$ overlaps with the interval $[b_1,b^{\eps}_1]$.  Recalling the notation $\bm{z}_{-i}$ we introduce the disjoint sets
\begin{align*}
&\Theta^\eps_n(x_1):=\{\bm{z}_{-1}:x_1\geq b^{\eps}_1(t,\bm z_{-1})\}\cup\{\bm{z}_{-1}:x_1+\tfrac{1}{n}\leq b_1(t,\bm z_{-1})\},\\
&\Gamma^\eps_n(x_1):=\{\bm{z}_{-1}:x_1+\tfrac{1}{n}\geq b^{\eps}_1(t,\bm z_{-1})\}\cap\{\bm{z}_{-1}:b_1(t,\bm z_{-1})\ge x_1\},\\
\end{align*}
and
\begin{align*}
\Sigma^\eps_n(x_1):=&\big\{\bm{z}_{-1}: x_1+\tfrac{1}{n}\geq b^{\eps}_1(t,\bm z_{-1})>x_1>b_1(t,\bm z_{-1})\big\}\\
&\cup\big\{\bm{z}_{-1}: b^{\eps}_1(t,\bm z_{-1})>x_1+\tfrac{1}{n}>x_1>b_1(t,\bm z_{-1})\big\}\\
&\cup\big\{\bm{z}_{-1}: b^{\eps}_1(t,\bm z_{-1})>x_1+\tfrac{1}{n}>b_1(t,\bm z_{-1})\ge x_1\big\}\\
=:&\,\Sigma^\eps_{n,1}(x_1)\cup\Sigma^\eps_{n,2}(x_1)\cup \Sigma^\eps_{n,3}(x_1).
\end{align*}
Given that $t$ is fixed and we only integrate in the spatial variables in \eqref{eq:un_ij}, we omit $t$ from the notation for the sets $\Theta^\eps_n$, $\Gamma^\eps_n$ and $\Sigma^\eps_n$. It is useful to observe that 
\begin{align}\label{eq:partition}
\Lambda_n^{-1}(\bm{x})=\Theta^\eps_n(x_1)\cup\Gamma^\eps_n(x_1)\cup \Sigma^\eps_n(x_1).
\end{align}
So the integral \eqref{eq:un_ij} can be written as
\begin{equation}
\label{eq:un1}
\begin{aligned}
U^n_{x_1 x_j}(t,\bm x)&=n^m\int_{\Lambda_n^{-1}(\bm x)}\big[ U_{x_j}(t,x_1+\tfrac{1}{n},\bm{z}_{-1})-U_{x_j}(t,x_1,\bm{z}_{-1})\big]\d\bm{z}_{-1}\\
&=n^m\int_{\Theta^\eps_n(x_1)} \Big(\int_{x_1}^{x_1+\tfrac{1}{n}}U_{x_1 x_j}(t,z_1,\bm{z}_{-1})\d z_1\Big)\d \bm{z}_{-1}\\
&\quad+n^m\int_{\Gamma^\eps_n(x_1)}\big[ U_{x_j}(t,x_1+\tfrac{1}{n},\bm{z}_{-1})-U_{x_j}(t,x_1,\bm{z}_{-1})\big]\d\bm{z}_{-1}\\
&\quad+n^m\int_{\Sigma^\eps_n(x_1)}\left[U_{x_j}(t,x_1+\tfrac{1}{n},\bm z_{-1})-U_{x_j}(t,x_1,\bm z_{-1})\right]\d\bm z_{-1},
\end{aligned}
\end{equation}
where we also used that $U_{x_1 x_j}$ is continuous on $[x_1,x_1+\frac{1}{n}]\times \Theta^\eps_n(x_1)$. In the first integral (on the set $\Theta^\eps_n(x_1)$) we have
\begin{equation}\label{eq:th}
\begin{aligned}
&n^m\int_{\Theta^\eps_n(x_1)} \Big(\int_{x_1}^{x_1+\tfrac{1}{n}}U_{x_1 x_j}(t,\bm{z})\d z_1\Big)\d \bm{z}_{-1}\\
&=n^m\int_{\Lambda^{-1}_n(\bm{x})} \mathds{1}_{\{x_1\ge b^\eps_1(t,\bm{z}_{-1})\}}\Big(\int_{x_1}^{x_1+\tfrac{1}{n}}\mathds{1}_{\{z_1\ge b^\eps_1(t,\bm{z}_{-1})\}}U_{x_1 x_j}(t,\bm{z})\d z_1\Big)\d \bm{z}_{-1}\\
&\quad+n^m\int_{\Lambda^{-1}_n(\bm{x})} \mathds{1}_{\{x_1+\frac{1}{n}\le b_1(t,\bm{z}_{-1})\}}\Big(\int_{x_1}^{x_1+\tfrac{1}{n}}\mathds{1}_{\{z_1\le b_1(t,\bm{z}_{-1})\}}U_{x_1 x_j}(t,\bm{z})\d z_1\Big)\d \bm{z}_{-1}.
\end{aligned}
\end{equation}

In the second integral (on the set $\Gamma^\eps_n(x_1)$) we can add and subtract $U_{x_j}(t,b^\eps_1(t,\bm{z}_{-1}),\bm{z}_{-1})$ and $U_{x_j}(t,b_1 (t,\bm{z}_{-1}),\bm{z}_{-1})$ to obtain
\begin{equation}\label{eq:Ga}
\begin{aligned}
&n^m\int_{\Gamma^\eps_n(x_1)}\big[ U_{x_j}(t,x_1+\tfrac{1}{n},\bm{z}_{-1})-U_{x_j}(t,x_1,\bm{z}_{-1})\big]\d\bm{z}_{-1}\\
&=n^m\int_{\Gamma^\eps_n(x_1)}\Big(\int_{b^\eps_1(t,\bm{z}_{-1})}^{x_1+\tfrac{1}{n}}U_{x_1 x_j}(t,\bm{z})\d z_1\Big)\d \bm{z}_{-1}\\
&\quad+\!n^m\!\!\int_{\Gamma^\eps_n(x_1)}\!\!\big[ U_{x_j}(t,b^\eps_1(t,\bm{z}_{-1}),\bm{z}_{-1})\!-\!U_{x_j}(t,b_1 (t,\bm{z}_{-1}),\bm{z}_{-1})\big]\d\bm{z}_{-1}\\
&\quad+n^m\int_{\Gamma^\eps_n(x_1)}\Big(\int^{b_1(t,\bm{z}_{-1})}_{x_1}U_{x_1 x_j}(t,\bm{z})\d z_1\Big)\d \bm{z}_{-1}
\end{aligned}
\end{equation}
by using that $U_{x_1 x_j}$ is continuous in $\overline{\cC^\eps_1}$ and in $\cD$. In the third integral (on the set $\Sigma^\eps_n(x_1)$) we can also proceed in a similar way taking advantage of the decomposition over $\Sigma^\eps_{n,1}(x_1)$, $\Sigma^\eps_{n,2}(x_1)$ and $\Sigma^\eps_{n,3}(x_1)$. In particular, that gives
\begin{equation}\label{eq:Si1}
\begin{aligned}
&n^m\int_{\Sigma^\eps_{n,1}(x_1)}\left[U_{x_j}(t,x_1+\tfrac{1}{n},\bm z_{-1})-U_{x_j}(t,x_1,\bm z_{-1})\right]\d\bm z_{-1}\\
&=n^m\int_{\Sigma^\eps_{n,1}(x_1)}\Big(\int^{x_1+\tfrac{1}{n}}_{b^\eps_1(t,\bm{z}_{-1})}U_{x_1 x_j}(t,\bm{z})\d z_1\Big)\d \bm{z}_{-1}\\ 
&\quad+n^m\int_{\Sigma^\eps_{n,1}(x_1)}\left[U_{x_j}(t,b^\eps_1(t,\bm{z}_{-1}),\bm z_{-1})-U_{x_j}(t,x_1,\bm z_{-1})\right]\d\bm z_{-1}
\end{aligned}
\end{equation}
and
\begin{equation}\label{eq:Si2}
\begin{aligned}
&n^m\int_{\Sigma^\eps_{n,3}(x_1)}\left[U_{x_j}(t,x_1+\tfrac{1}{n},\bm z_{-1})-U_{x_j}(t,x_1,\bm z_{-1})\right]\d\bm z_{-1}\\
&=n^m\int_{\Sigma^\eps_{n,3}(x_1)}\left[U_{x_j}(t,x_1+\tfrac{1}{n},\bm z_{-1})-U_{x_j}(t,b_1(t,\bm{z}_{-1}),\bm z_{-1})\right]\d\bm z_{-1}\\ 
&\quad+n^m\int_{\Sigma^\eps_{n,3}(x_1)}\Big(\int^{b_1(t,\bm{z}_{-1})}_{x_1}U_{x_1 x_j}(t,\bm{z})\d z_1\Big)\d \bm{z}_{-1}.
\end{aligned}
\end{equation}
Let us notice that we can add up the first term on the right-hand side of \eqref{eq:th}, \eqref{eq:Ga} and \eqref{eq:Si1}, which gives
\begin{equation*}%\label{eq:P1}
\begin{aligned}
&n^m\!\int_{\Lambda^{-1}_n(\bm x)} \mathds{1}_{\{x_1\ge b^\eps_1(t,\bm{z}_{-1})\}}\Big(\int_{x_1}^{x_1+\tfrac{1}{n}}\mathds{1}_{\{z_1\ge b^\eps_1(t,\bm{z}_{-1})\}}U_{x_1 x_j}(t,\bm{z})\d z_1\Big)\d \bm{z}_{-1}\\
&+n^m\int_{\Gamma^\eps_n(x_1)}\Big(\int_{b^\eps_1(t,\bm{z}_{-1})}^{x_1+\tfrac{1}{n}}U_{x_1 x_j}(t,\bm{z})\d z_1\Big)\d \bm{z}_{-1}+n^m\int_{\Sigma^\eps_{n,1}(x_1)}\Big(\int^{x_1+\tfrac{1}{n}}_{b^\eps_1(t,\bm{z}_{-1})}U_{x_1 x_j}(t,\bm{z})\d z_1\Big)\d \bm{z}_{-1}.
\end{aligned}
\end{equation*}
The above expression is equal to  

\begin{equation}\label{eq:P1}
\begin{aligned}
&n^m\int_{\Lambda^{-1}_n(\bm x)} \mathds{1}_{\{x_1\ge b^\eps_1(t,\bm{z}_{-1})\}}\Big(\int_{x_1}^{x_1+\tfrac{1}{n}}\mathds{1}_{\{z_1\ge b^\eps_1(t,\bm{z}_{-1})\}}U_{x_1 x_j}(t,\bm{z})\d z_1\Big)\d \bm{z}_{-1}\\
&\quad+n^m\int_{\Lambda^{-1}_n(\bm x)}\mathds{1}_{\{x_1+\tfrac{1}{n}\ge b^\eps_1(t,\bm{z}_{-1})> x_1\}}\Big(\int_{x_1}^{x_1+\tfrac{1}{n}}\mathds{1}_{\{z_1\ge b^\eps_1(t,\bm{z}_{-1})\}}U_{x_1 x_j}(t,\bm{z})\d z_1\Big)\d \bm{z}_{-1}\\
&=n^m\!\int_{\Lambda^{-1}_n(\bm x)} \Big(\int_{x_1}^{x_1+\tfrac{1}{n}}\mathds{1}_{\{z_1\ge b^\eps_1(t,\bm{z}_{-1})\}}U_{x_1 x_j}(t,\bm{z})\d z_1\Big)\d \bm{z}_{-1}\\
&=n^m\!\int_{\Lambda_n(\bm x)} \mathds{1}_{\{z_1\ge b^\eps_1(t,\bm{z}_{-1})\}}U_{x_1 x_j}(t,\bm{z})\d \bm{z},
\end{aligned}
\end{equation}
where the first equality uses the fact that on $\{x_1+\tfrac{1}{n}<b^\eps_1(t,\bm{z}_{-1})\}$ the integral with respect to $\d z_1$ vanishes.
Similarly, we can now add up the second term on the right-hand side of \eqref{eq:th} and \eqref{eq:Si2} with the third one on the right-hand side of \eqref{eq:Ga}, to obtain
\begin{equation}\label{eq:P2}
\begin{aligned}
n^m&\int_{\Lambda^{-1}_n(\bm x)} \mathds{1}_{\{x_1+\frac{1}{n}\le b_1(t,\bm{z}_{-1})\}}\Big(\int_{x_1}^{x_1+\tfrac{1}{n}}\mathds{1}_{\{z_1\le b_1(t,\bm{z}_{-1})\}}U_{x_1 x_j}(t,\bm{z})\d z_1\Big)\d \bm{z}_{-1}\\
&+n^m\int_{\Gamma^\eps_n(x_1)}\Big(\int^{b_1(t,\bm{z}_{-1})}_{x_1}U_{x_1 x_j}(t,\bm{z})\d z_1\Big)\d \bm{z}_{-1}\\
&+n^m\int_{\Sigma^\eps_{n,3}(x_1)}\Big(\int^{b_1(t,\bm{z}_{-1})}_{x_1}U_{x_1 x_j}(t,\bm{z})\d z_1\Big)\d \bm{z}_{-1}\\
=&\,n^m\!\int_{\Lambda_n(\bm x)} \mathds{1}_{\{z_1\le b_1(t,\bm{z}_{-1})\}}U_{x_1 x_j}(t,\bm{z})\d \bm{z}.
\end{aligned}
\end{equation}
Finally, we gather the remaining terms from \eqref{eq:Ga}, \eqref{eq:Si1}, \eqref{eq:Si2} and the one remaining integral from \eqref{eq:un1} (i.e., the one over $\Sigma^\eps_{n,2}(x_1)$) and denote
\begin{equation}\label{eq:P3}
\begin{aligned}
F^{n,\eps}_{1j}(t,\bm x):=&n^m\!\!\int_{\Gamma^\eps_n(x_1)}\!\!\big[ U_{x_j}(t,b^\eps_1(t,\bm{z}_{-1}),\bm{z}_{-1})\!-\!U_{x_j}(t,b_1 (t,\bm{z}_{-1}),\bm{z}_{-1})\big]\d\bm{z}_{-1}\\
&+n^m\int_{\Sigma^\eps_{n,1}(x_1)}\left[U_{x_j}(t,b^\eps_1(t,\bm{z}_{-1}),\bm z_{-1})-U_{x_j}(t,x_1,\bm z_{-1})\right]\d\bm z_{-1}\\
&+n^m\int_{\Sigma^\eps_{n,2}(x_1)}\left[U_{x_j}(t,x_1+\tfrac{1}{n},\bm z_{-1})-U_{x_j}(t,x_1,\bm z_{-1})\right]\d\bm z_{-1}
\\
&+n^m\int_{\Sigma^\eps_{n,3}(x_1)}\left[U_{x_j}(t,x_1+\tfrac{1}{n},\bm z_{-1})-U_{x_j}(t,b_1(t,\bm{z}_{-1}),\bm z_{-1})\right]\d\bm z_{-1}.
\end{aligned}
\end{equation}

Combining \eqref{eq:P1}, \eqref{eq:P2} and \eqref{eq:P3} we obtain \eqref{eq:goal-1} for $i=1$. Before proving that indeed $F^{n,\eps}_{1j}$ vanishes as $\eps\downarrow 0$ while keeping $n$ fixed, we prove \eqref{eq:goal-1} for a generic couple $i,j$.

Fix $i\neq 1,j\neq 1$ and recall that we are currently assuming $b_1$ non-decreasing in all its arguments. Then, in particular we can define the generalised (left-continuous) inverse of $b_1$ with respect to $x_i$:
\begin{equation}
\label{eq:invb_1}
b_i(t,\bm{x}_{-i}):=\sup\{x_i\in\R : x_1>b_1(t,x_2,\ldots,x_m)\}.
\end{equation}
It is not hard to check that $x_1>b_1(t,\bm{x}_{-1})\iff x_i<b_i(t,\bm{x}_{-i})$, $x_1\mapsto b_i(t,\bm{x}_{-i})$ is non-decreasing, while $x_j\mapsto b_i(t,\bm{x}_{-i})$ and $t\mapsto b_i(t,\bm{x}_{-i})$ are non-increasing for all $j\neq\{1,i\}$. Thus, we can parametrise $\cC$ and $\cD$ as 
\begin{equation}\label{eq:pCD}
\begin{aligned}
&\cC=\{(t,\bm x)\in \R_+\times\R^m : x_i< b_i(t,\bm{x}_{-i})\},\\
&\cD=\{(t,\bm x)\in \R_+\times\R^m : x_i\ge b_i(t,\bm{x}_{-i})\},
\end{aligned}
\end{equation}
and the analogue of \eqref{eq:b-1-eps} in this case is
\begin{align}\label{eq:bei}
b^\eps_i(t,\bm{x}_{-i}):=b_i(t+\eps,x_1-\eps,x_2+\eps,\ldots, x_{i-1}+\eps,x_{i+1}+\eps, \ldots, x_m+\eps)-\eps.
\end{align}
It is important to notice that, thanks to the monotonicity stated above for $b^\eps_i$, the limit:
\[
b^{0+}_i(t,\bm{x}_{-i}):=\lim_{\eps\downarrow 0}b^{\eps}_i(t,\bm{x}_{-i})
\]
exists and an $b^{0+}_i(t,\bm{x}_{-i})\le b_i(t,\bm{x}_{-i})$. Then, as in the case of $b^\eps_1$ above, since $\cD$ is closed we have 
\[
(t,x_1,\ldots,x_{i-1},b^{0+}_i(t,\bm{x}_{-i}),x_{i+1},\ldots x_m)\in\cD,
\]
Hence 
\begin{align}\label{eq:limbe}
b^{0+}_i(t,\bm{x}_{-i})\le b_i(t,\bm{x}_{-i})\le b^{0+}_i(t,\bm{x}_{-i}).
\end{align}
  Furthermore, letting 
\begin{align}\label{eq:C1ei}
\cC^\eps_i:=\{(t,\bm x)\in\R_+\times\R^m : x_i<b^\eps_{i}(t,\bm x_{-i})\}
\end{align}
we have $\overline {\cC^\eps_i}\subset \cC$, for all $\eps>0$. Thus, repeating the same estimates as above we obtain 
\begin{equation}\label{eq:goalD}
\begin{aligned}
U^n_{x_i x_j}(t,\bm x)=&n^m\int_{\Lambda_n(\bm x)}U_{x_i x_j}(t,\bm{z})\mathds{1}_{\{z_i\leq b^{\eps}_i(t,\bm z_{-i})\}}\d\bm{z}\\
&+n^m\int_{\Lambda_n(\bm x)}U_{x_i x_j}(t,\bm{z})\mathds{1}_{\{z_i\ge b_i(t,\bm z_{-i})\}}\d\bm{z}
+F^{n,\eps}_{ij}(t,\bm x), 
\end{aligned}
\end{equation}
where
\begin{equation}\label{eq:P4-1}
\begin{aligned}
F^{n,\eps}_{ij}(t,\bm x):=&n^m\!\!\int_{\Gamma^\eps_n(x_i)}\!\!\big[ U_{x_j}(t,b_i(t,\bm{z}_{-i}),\bm{z}_{-i})\!-\!U_{x_j}(t,b^\eps_i (t,\bm{z}_{-i}),\bm{z}_{-i})\big]\d\bm{z}_{-i}\\
&+n^m\int_{\Sigma^\eps_{n,1}(x_i)}\left[U_{x_j}(t,x_i+\tfrac{1}{n},\bm z_{-i})-U_{x_j}(t,b^\eps_i(t,\bm{z}_{-i}),\bm z_{-i})\right]\d\bm z_{-i}\\
&+n^m\int_{\Sigma^\eps_{n,2}(x_i)}\left[U_{x_j}(t,x_i+\tfrac{1}{n},\bm z_{-i})-U_{x_j}(t,x_i,\bm z_{-i})\right]\d\bm z_{-i}\\
&+n^m\int_{\Sigma^\eps_{n,3}(x_i)}\left[U_{x_j}(t,b_i(t,\bm{z}_{-i}),\bm z_{-i})-U_{x_j}(t,x_i,\bm z_{-i})\right]\d\bm z_{-i}
\end{aligned}
\end{equation}
and we have substituted the sets $\Gamma^\eps_n$, $\Sigma^\eps_{n,1}$, $\Sigma^\eps_{n,2}$ and $\Sigma^\eps_{n,3}$ from \eqref{eq:P3} with their counterparts in this case:
\begin{align*}
&\Theta^\eps_n(x_i):=\{\bm{z}_{-i}:x_i+\tfrac{1}{n}\leq b^{\eps}_i(t,\bm z_{-i})\}\cup\{\bm{z}_{-i}:x_i\geq b_i(t,\bm z_{-i})\},\\
&\Gamma^\eps_n(x_i):=\{\bm{z}_{-i}:x_i\leq b^{\eps}_i(t,\bm z_{-i})\}\cap\{\bm{z}_{-i}:x_i+\tfrac{1}{n}\ge b_i(t,\bm z_{-i})\},\\
\end{align*}
and
\begin{align*}
\Sigma^\eps_n(x_i):=&\big\{\bm{z}_{-i}: x_i\leq b^{\eps}_i(t,\bm z_{-i})<x_i+\tfrac{1}{n}<b_i(t,\bm z_{-i})\big\}\\
&\cup\big\{\bm{z}_{-i}: b^{\eps}_i(t,\bm z_{-i})<x_i<x_i+\tfrac{1}{n}<b_i(t,\bm z_{-i})\big\}\\
&\cup\big\{\bm{z}_{-i}: b^{\eps}_i(t,\bm z_{-i})<x_i<b_i(t,\bm z_{-i})\le x_i+\tfrac{1}{n}\big\}\\
=:&\,\Sigma^\eps_{n,1}(x_i)\cup\Sigma^\eps_{n,2}(x_i)\cup \Sigma^\eps_{n,3}(x_i).
\end{align*} 

The sets $\{z_i=b^\eps_i(t,\bm z_{-i})\}$ and $\{z_i=b_i(t,\bm z_{-i})\}$ have zero Lebesgue measure in $\R^m$, so that we can take strict inequalities in the indicator functions in the integrals in \eqref{eq:goalD}. Then we can also use the equivalences
\begin{align}\label{eq:inverse}
z_i<b_i(t,\bm z_{-i})\iff z_1> b_1(t,\bm z_{-1})
\end{align}
and
\begin{align*}
&z_i<b^\eps_i(t,\bm z_{-i})\iff z_i+\eps<b_i(t+\eps,z_1-\eps, z_2+\eps,\ldots z_m+\eps)\\
&\iff z_1-\eps> b_1(t+\eps, z_2+\eps,z_3+\eps,\ldots z_m+\eps)\iff z_1>b^\eps_1(t,\bm z_{-1}),
\end{align*}
to rewrite \eqref{eq:goalD} as
\begin{equation*}%\label{eq:goalb}
\begin{aligned}
U^n_{x_i x_j}(t,\bm x)=&n^m\int_{\Lambda_n(\bm x)}U_{x_i x_j}(t,\bm{z})\mathds{1}_{\{z_1\ge b^{\eps}_1(t,\bm z_{-1})\}}\d\bm{z}\\
&+n^m\int_{\Lambda_n(\bm x)}U_{x_i x_j}(t,\bm{z})\mathds{1}_{\{z_1\le b_1(t,\bm z_{-1})\}}\d\bm{z}
+F^{n,\eps}_{ij}(t,\bm x).
\end{aligned}
\end{equation*}
This proves \eqref{eq:goal-1} for arbitrary $i,j$.

{\em Step 2}. Now that we have derived \eqref{eq:goal-1} we are in a position to find the bound \eqref{eq:goal}. To keep the notation simple, below we write $\Lambda_n=\Lambda_n(\bm{x})$ since $\bm{x}$ is fixed and no confusion shall arise.
Indeed, we have
\begin{equation}
\label{eq:bd-1}
\begin{aligned}
&\sum_{i,j=1}^m\beta^{ij}(t,\bm x)U^n_{x_i x_j}(t,\bm x)\\
&=n^m\int_{\Lambda_n}\sum^{m}_{i,j=1}\beta^{ij}(t,\bm z)U_{x_i x_j}(t,\bm z)\mathds{1}_{\{z_1\ge b^{\eps}_1(t,\bm z_{-1})\}\cup\{z_1\le b_1(t,\bm z_{-1})\}}\d\bm z\\
&\quad+n^m\int_{\Lambda_n}\sum^{m}_{i,j=1}\big(\beta^{ij}(t,\bm x)-\beta^{ij}(t,\bm z)\big)U_{x_i x_j}(t,\bm z)\mathds{1}_{\{z_1\ge b^{\eps}_1(t,\bm z_{-1})\}\cup\{z_1\le b_1(t,\bm z_{-1})\}}\d\bm z\\
&\quad+\sum^{m}_{i,j=1}\beta^{ij}(t,\bm x)F^{n,\eps}_{ij}(t,\bm x).
\end{aligned}
\end{equation}
Thanks to Assumption \ref{ass:1-2}, there exists $c_{1,\delta}>0$, depending only on the compact $V^\delta$ in \eqref{eq:V-delta}, such that 
\begin{align}\label{eq:BDD0}
&\bigg|n^m\int_{\Lambda_n}\sum^{m}_{i,j=1}\beta^{ij}(t,\bm z)U_{x_i x_j}(t,\bm z)\mathds{1}_{\{z_1\ge b^{\eps}_1(t,\bm z_{-1})\}\cup\{z_1\le b_1(t,\bm z_{-1})\}}\d\bm z\bigg|\le n^m\int_{\Lambda_n}\!\!c_{1,\delta}\, \d \bm z=c_{1,\delta}.
\end{align}
Moreover, recalling that $\cD$ is closed, $\beta^{ij}$ is continuous and $U\in C^{1,2}(\cD)$ we also have 
\begin{align}\label{eq:BDD1}
\bigg|n^m\int_{\Lambda_n}\sum^{m}_{i,j=1}\big(\beta^{ij}(t,\bm x)-\beta^{ij}(t,\bm z)\big)U_{x_i x_j}(t,\bm z)\mathds{1}_{\{z_1\le b_1(t,\bm z_{-1})\}}\d\bm z\bigg|\le n^m\int_{\Lambda_n}c_{2,\delta}\, \d \bm z=c_{2,\delta},
\end{align}
for some other constant $c_{2,\delta}>0$ only depending on $V^\delta$.

Next we find a bound for the second integral on the right-hand side of \eqref{eq:bd-1} on the indicator of the set $\{z_1\ge b^\eps_1(t,\bm z_{-1})\}$. We provide the details for $i\neq 1$, $j \neq 1 $, but it will be clear that the same arguments apply for $i=1$ and/or $j=1$. Recalling \eqref{eq:inverse} and the discussion following that expression we have
\begin{equation}
\label{eq:bd-2}
\begin{aligned}
&n^m\int_{\Lambda_n}\big(\beta^{ij}(t,\bm x)-\beta^{ij}(t,\bm z)\big)U_{x_i x_j}(t,\bm z)\mathds{1}_{\{z_1\ge b^{\eps}_1(t,\bm z_{-1})\}}\d\bm z\\
&=n^m\int_{\Lambda_n}\big(\beta^{ij}(t,\bm x)-\beta^{ij}(t,\bm z)\big)U_{x_i x_j}(t,\bm z)\mathds{1}_{\{z_i\le b^{\eps}_i(t, \bm z_{-i})\}}\d\bm z\\
&=n^m\int_{\Lambda_n^{-1}}1_{\{x_i\le b^\eps_i(t,\bm z_{-i})\}}\Big(\int^{b^{\eps}_i(t, \bm z_{-i})\wedge(x_i+\frac{1}{n})}_{x_i}\!\big(\beta^{ij}(t,\bm x)-\beta^{ij}(t,\bm z)\big)U_{x_i x_j}(t,\bm z)\d z_i\Big)\d\bm z_{-i}.
\end{aligned}
\end{equation}

By Assumption \ref{ass:1-2} we know there is a constant $\kappa_{\delta}>0$ such that $\sup_{V^\delta}\sum_{j=1}^m|U_{x_j}|\le \kappa_\delta$. Integrating by parts with respect to $z_i$ and recalling that $\beta^{ij}$ is locally Lipschitz (hence Lipschitz on $V^\delta$ with constant $L_{\beta,\delta}>0$ which can be taken independent of $i,j$) gives 
\begin{align*}
&\left|\int^{b^{\eps}_i(t, \bm z_{-i})\wedge (x_i+\frac{1}{n})}_{x_i}\!\big(\beta^{ij}(t,\bm x)-\beta^{ij}(t,\bm z)\big)U_{x_i x_j}(t,\bm z)\d z_i\right|\\
&=\left|\Big[\big(\beta^{ij}(t,\bm x)-\beta^{ij}(t,\bm z)\big)U_{x_j}(t,\bm z)\Big]^{z_i=b^{\eps}_i(t, \bm z_{-i})\wedge (x_i+\frac{1}{n})}_{z_i=x_i}+\int^{b^{\eps}_i(t, \bm z_{-i})\wedge (x_i+\frac{1}{n})}_{x_i}\!\beta^{ij}_{x_i}(t,\bm z)U_{x_j}(t,\bm z)\d z_i\right|\\
&\le 2\kappa_\delta L_{\beta,\delta}\frac{\sqrt{m}}{n}+\kappa_\delta L_{\beta,\delta}\frac{1}{n}=:c_{3,\delta}\frac{1}{n},
\end{align*}
upon using that the Euclidean norm $\|\bm x-\bm z\|\le \sqrt{m}/n$ for all $\bm z\in\Lambda_n$ and, in particular, $|x_i-b^{\eps}_i(t, \bm z_{-i})\wedge (x_i+\tfrac{1}{n})|\le 1/n$. 

Plugging the above bound back into \eqref{eq:bd-2} we obtain
\begin{equation}\label{eq:bDD}
\begin{aligned}
&n^m\!\!\int_{\Lambda_n}\!\!\big(\beta^{ij}(t,\bm x)-\beta^{ij}(t,\bm z)\big)U_{x_i x_j}(t,\bm z)\mathds{1}_{\{z_1\ge b^{\eps}_1(t,\bm z_{-1})\}}\d\bm z\le c_{3,\delta}\, n^{m-1}\!\!\int_{\Lambda_n^{-1}}\d \bm z_{-1}= c_{3,\delta}.
\end{aligned}
\end{equation}

Thanks to \eqref{eq:bd-1}, \eqref{eq:BDD0}, \eqref{eq:BDD1} and \eqref{eq:bDD} we have
\begin{equation}
\label{eq:bd-7}
\begin{aligned}
&\bigg|\sum_{i,j}\beta^{ij}(t,\bm x)U^n_{x_i x_j}(t,\bm x)\bigg|\le c_{1,\delta}+c_{2,\delta}+m^2 c_{3,\delta} +\bigg|\sum_{i,j}\beta^{ij}(t,\bm x)F^{n,\eps}_{ij}(t,\bm x)\bigg|,
\end{aligned}
\end{equation}
for all $(t,\bm x)\in V^\delta$. Finally, letting $\eps\downarrow 0$ and using that $U\in C^1 (\R_+\times\R^m)$ and the convergence of $b^{\eps}_i$ to $b_i$ for all $i$'s (recall \eqref{eq:limbe}), we obtain 
\[
\lim_{\eps\downarrow 0} F^{n,\eps}_{ij}(t,\bm x)= 0.
\]
Hence
\[
\bigg|\sum_{i,j}\beta^{ij}(t,\bm x)U^n_{x_i x_j}(t,\bm x)\bigg|\le c_{1,\delta}+c_{2,\delta}+m^2 c_{3,\delta},\quad\text{for all $(t,\bm x)\in V^\delta$}.
\]
The latter is equivalent to \eqref{eq:goal} with $C_\delta:=c_{1,\delta}+c_{2,\delta}+m^2 c_{3,\delta}$, since the constants are independent of $(t,\bm x)\in V^\delta$.

This completes the proof of the theorem in the case \eqref{eq:monot} holds.\hfill $\square$

\subsection{Relaxing condition \eqref{eq:monot}}\label{sec:med}
The case in which the boundary has different monotonicity in each variable (as allowed by Assumption \ref{ass:1-3}) can be addressed by the same methods employed above up to some obvious changes. In order to illustrate the main points, fix $2\le \bar k\le m$ and let us assume with no loss of generality that $t\mapsto b_1(t,\bm x_{-1})$ and $x_i\mapsto b_1(t,\bm x_{-1})$ are non-decreasing for $2\le i\le \bar k$, while $x_i\mapsto b_1(t,\bm x_{-1})$ are non-increasing for $\bar k< i\le m$. Then, in the first part of step 1 in the proof above we replace \eqref{eq:b-1-eps} by
\begin{equation*}
%\label{eq:b-1-eps}
b^{\eps}_1(t,x_2,...x_m):=b_1(t+\eps,x_2+\eps,\ldots x_{\bar k}+\eps, x_{\bar k +1}-\eps,\ldots x_{m}-\eps)+\eps,
\end{equation*}
so that $b^{\eps}_1$ is decreasing as $\eps\downarrow 0$ and its limit $b^{0+}_1(t,\bm x_{-1})$ equals $b_1(t,\bm x_{-1})$ by closedness of $\cD$ and the same argument as in step 1. Also in this case \eqref{eq:C1eC} continues to hold and we can repeat verbatim the estimates that lead to \eqref{eq:goal-1} for $i=1$ in step 1 above. For the second part of step 1, we need the generalised inverse $b_i$ for each $i$. In particular, for $2\le i\le \bar k$ the same definition of $b_i$ as in \eqref{eq:invb_1} and the parametrisation of $\cC$ and $\cD$ as in \eqref{eq:pCD} continue to hold. However, $x_j\mapsto b_i(t,\bm{x}_{-i})$ is non-decreasing for $j=1$ and $\bar k < j \le m$, while $x_j\mapsto b_i(t,\bm{x}_{-i})$ and $t\mapsto b_i(t,\bm{x}_{-i})$ are non-increasing for all $2\le j\le \bar k$ with $j\neq i$. Then, setting 
\[
b^\eps_i(t,\bm x_{-i}):=b_i(t+\eps,x_1-\eps, x_2+\eps,\ldots x_{\bar k}+\eps, x_{\bar k+1}-\eps,\ldots x_m-\eps)-\eps
\] 
the functions $b^\eps_i$ increase as $\eps\downarrow 0$ and in the limit $b^{0+}_i(t,\bm x_{-i})$ equals $b_i(t,\bm x_{-i})$. So we can repeat the same arguments as in step 1 and obtain \eqref{eq:goal-1} for $2\le i\le \bar k$ and any $j$. Finally, for $\bar k< i\le m$, since $x_i\mapsto b_1(t,\bm x_{-1})$ is non-increasing we define its (left-continuous)
generalised inverse as 
\[
b_i(t,\bm x_{-i}):=\inf\{x_i\in\R: x_1>b_1(t,\bm x_{-1})\}.
\]
Then we have $x_1> b_1(t,\bm x_{-1})\iff x_i> b_i(t,\bm x_{-i})$, $t\mapsto b_i(t,\bm x_{-i})$ and $x_j\mapsto b_i(t,\bm x_{-i})$ are non-decreasing for $2\le j\le \bar k$, while $x_1\mapsto b_i(t,\bm x_{-i})$ and $x_j\mapsto b_i(t,\bm x_{-i})$ are non-increasing for $\bar k< j\le m$ with $j\neq i$. 
The sets $\cC$ and $\cD$ can be parametrised as 
\begin{equation*}%\label{eq:pCD}
\begin{aligned}
&\cC=\{(t,\bm x)\in \R_+\times\R^m : x_i> b_i(t,\bm{x}_{-i})\},\\
&\cD=\{(t,\bm x)\in \R_+\times\R^m : x_i\le b_i(t,\bm{x}_{-i})\},
\end{aligned}
\end{equation*}
and we can define the functions 
\[
b^\eps_i(t,\bm x_{-i}):=b_i(t+\eps,x_1-\eps, x_2+\eps,\ldots x_{\bar k}+\eps, x_{\bar k+1}-\eps,\ldots x_m-\eps)+\eps.
\] 
The latter decrease as $\eps\downarrow 0$ and converge to $b_i(t,\bm x_{-i})$ by closedness of $\cD$. Since once again $\overline{\cC^\eps_i}\subset \cC$, we can repeat the arguments from step 1 and arrive at \eqref{eq:goal-1} also for all $j$'s and $i\neq 1$. 

This completes the analogy with step 1. Step 2 can be repeated verbatim. Thus the theorem holds under the generality of Assumption \ref{ass:1-3} concerning the boundary.\hfill $\square$

\appendix
\section{}

In this appendix we provide a comparison result for 1-dimensional stochastic differential equations with random coefficients. The next proposition is essentially a small variation on \cite[Prop.\ 5.2.18]{KS} which holds for deterministic coefficients.
\begin{proposition}\label{prop:comparison}
Fix a filtered probability space $(\Omega,\cF,(\cF_t)_{t\ge 0},\P)$ equipped with a Brownian motion $(B_t)_{t\ge 0}$. Let $\theta:\Omega\times[0,T]\times\R\to\R_+$ and $\eta^i:\Omega\times[0,T]\times\R\to\R$, $i=1,2$ be measurable functions such that
\begin{itemize}
\item[(i)] $(\omega,t)\mapsto\theta(\omega,t,y)$ and $(\omega,t)\mapsto\eta^i(\omega,t,y)$ are progressively measurable for every $y\in\R$;
\item[(ii)] There is a constant $K$ and a function $h:\R_+\to\R_+$ such that 
\begin{align*}
&|\eta^i(\omega,t,y)-\eta^i(\omega,t,y')|\le K|y-y'|,\quad\text{for either $i=1$ or $i=2$},\\
&|\theta(\omega,t,y)-\theta(\omega,t,y')|\le h(|y-y'|),
\end{align*} 
and the function $h$ is such that 
\[
\int_0^\eps h^{-2}(u)\d u=\infty,\quad\text{for every $\eps>0$}.
\]
\end{itemize}
Let $C$ be a c\`adl\`ag process of bounded variation. For $i=1,2$ let $Y^i_0$ be $\cF_0$-measurable and $Y^i$ be the unique $(\cF_t)$-adapted solution of the SDE
\[
Y^i_t=Y^i_0+\int_0^t \eta^i(\omega,s,Y^i_s)\d s+ \int_0^t \theta(\omega,s,Y^i_s)\d B_s+ C_t.
\] 
If $\eta^1(\omega,t,y)\le \eta^2(\omega,t,y)$ for all $(\omega,t,y)\in\Omega\times[0,T]\times\R$ and $Y^1_0\le Y^2_0$, $\P$-a.s., then $Y^1_t \le Y^2_t$ for all $t\ge 0$, $\P$-a.s.
\end{proposition}
\begin{proof}
Setting $\bar Y^i_t:=Y^i_t-C_t$, $\bar \eta^i (\omega,t,y):= \eta^i (\omega,t,y+C_t(\omega))$ and $\bar \theta (\omega,t,y):= \theta (\omega,t,y+C_t(\omega))$ it is clear that $\bar Y^i$ solves
\[
\bar Y^i_t=Y^i_0+\int_0^t \bar \eta^i(\omega,s,\bar Y^i_s)\d s+ \int_0^t \bar \theta(\omega,s,\bar Y^i_s)\d B_s.
\] 
Now we can apply the same line of proof as in \cite[Prop.\ 5.2.18]{KS}, where it is shown that it is possible to construct a sequence of positive functions $(\varphi_n)_{n\in\N}\subset C^2(\R)$ such that  
\begin{equation}\label{eq:phin}
\begin{split}
&\text{$\varphi_n(y)\to(y)^+$ as $n\to\infty$ for all $y\in\R$,}\\
&\text{$|\varphi'_n(y)|\le 1$ with $\varphi'_n(y)=0$ for $y\le 0$ and $\varphi'_n(y)\ge 0$ for $y>0$},\\ 
&\text{$0\le \varphi_n(y)h^2(y)\le 2/n$ for $y>0$.}
\end{split}
\end{equation} 
Setting $Z_t:= Y^1_t-Y^2_t=\bar Y^1_t-\bar Y^2_t$ and applying It\^o's formula we obtain
\begin{align}\label{eq:KS}
\E[\varphi_n(Z_t)]=&\E\Big[\int_0^t \varphi'_n(Z_s)\big(\bar \eta^1(s,\bar Y^1_s)-\bar \eta^2(s,\bar Y^2_s)\big)\d s \Big]\\
&+\frac{1}{2}\E\Big[\int_0^t \varphi''_n(Z_s)\big(\bar \theta(s,\bar Y^1_s)-\bar \theta(s,\bar Y^2_s)\big)^2\d s \Big]\notag\\
\le &\E\Big[\int_0^t \varphi'_n(Z_s)1_{\{Z_s>0\}}\big(\bar \eta^1(s,\bar Y^1_s)-\bar \eta^2(s,\bar Y^2_s)\big)\d s \Big]+\frac{t}{n},\notag
\end{align}
where the inequality is by the third property in \eqref{eq:phin} and the indicator appears because of the second property in \eqref{eq:phin}. With no loss of generality we assume that $\eta^2$ is Lipschitz in the third variable, so that for $\bar \eta^2$ we have
\[
|\bar \eta^2(s,\bar Y^1_s)-\bar \eta^2(s,\bar Y^2_s)|\le K|Z_s|.
\]
Using that fact yields
\begin{align*}
&\varphi'_n(Z_s)1_{\{Z_s>0\}}\big(\bar \eta^1(s,\bar Y^1_s)-\bar \eta^2(s,\bar Y^2_s)\big)\\
&\le \varphi'_n(Z_s)1_{\{Z_s>0\}}\big(\bar \eta^2(s,\bar Y^1_s)-\bar \eta^2(s,\bar Y^2_s)\big)\le K \varphi'_n(Z_s)(Z_s)^+\le K (Z_s)^+.
\end{align*}
Plugging the latter back into \eqref{eq:KS} and letting $n\to\infty$ we obtain
\begin{align*}
\E[(Z_t)^+]\le K\int_0^t\E[(Z_s)^+]\d s.
\end{align*}
Gronwall's inequality gives $\P(Z_t=0)=1$ for all $t\ge 0$ and by continuity of $t\mapsto Z_t$ we conclude that $\P(Y^1_t\le Y^2_t,\,\forall t\ge 0)=1$.
\end{proof}
\medskip

\subsection*{Acknowledgements} We thank Goran Peskir for useful comments on the historical development of local time-space calculus and other generalisations of It\^o's formula. We are also grateful to Jan Palczewski for numerous discussions concerning technical aspects of the paper. Finally, we would like to also thank two anonymous referees whose insightful comments led to improvements in Section 3 of the paper.
Any inaccuracies or technical mistakes remain sole responsibility of the authors.

\bibliographystyle{abbrv} 
\bibliography{ref-ito}

\begin{thebibliography}{10}

\bibitem{alsmeyer2005useful}
G.~Alsmeyer and M.~Jaeger.
\newblock A useful extension of {I}t{\^o}’s formula with applications to
  optimal stopping.
\newblock {\em Acta Math.\ Sin.}, 21(4):779--786, 2005.

\bibitem{azema1998}
J.~Az{\'e}ma, T.~Jeulin, F.~B. Knight, and M.~Yor.
\newblock Quelques calculs de compensateurs impliquant l'injectivit{\'e} de
  certains processus croissants.
\newblock {\em S{\'e}m.\ {P}robab.\ {S}trasbourg}, 32:316--327, 1998.

\bibitem{bandini2020optimal}
E.~Bandini, T.~De~Angelis, G.~Ferrari, and F.~Gozzi.
\newblock Optimal dividend payout under stochastic discounting.
\newblock {\em Math. Finance}, 32(2):627--677, 2022.

\bibitem{bardina1997extension}
X.~Bardina and M.~Jolis.
\newblock An extension of {I}t{\^o}'s formula for elliptic diffusion processes.
\newblock {\em Stoch.\ Process.\ Appl.}, 69(1):83--109, 1997.

\bibitem{bensoussan}
A.~Bensoussan and J.-L. Lions.
\newblock {\em Applications of variational inequalities in stochastic control},
  volume~12 of {\em Studies in Mathematics and its Applications}.
\newblock North-Holland Publishing Co., Amsterdam-New York, 1982.
\newblock Translated from the French.

\bibitem{bouleau1981}
N.~Bouleau and M.~Yor.
\newblock Sur la variation quadratique des temps locaux de certaines
  semimartingales.
\newblock {\em C. R. Acad. Sci. Paris S\'{e}r. I Math.}, 292(9):491--494, 1981.

\bibitem{CDeAP2}
C.~Cai, T.~De~Angelis, and J.~Palczewski.
\newblock On the continuity of optimal stopping surfaces for jump diffusions.
\newblock {\em To appear in SIAM J. Control Optim. arXiv:2109.10810}, 2021.

\bibitem{CDeAP}
C.~Cai, T.~De~Angelis, and J.~Palczewski.
\newblock The {A}merican put with finite-time maturity and stochastic interest
  rate.
\newblock {\em Math.\ Finance.}, 32(4):1170 --1213, 2022.

\bibitem{CCMS}
S.~Christensen, F.~Crocce, E.~Mordecki, and P.~Salminen.
\newblock On optimal stopping of multidimensional diffusions.
\newblock {\em Stochastic Process. Appl.}, 129(7):2561--2581, 2019.

\bibitem{de2017dividend}
T.~De~Angelis and E.~Ekstr{\"o}m.
\newblock The dividend problem with a finite horizon.
\newblock {\em Ann. Appl. Probab.}, 27(6):3525--3546, 2017.

\bibitem{deangelis2017}
T.~De~Angelis, S.~Federico, and G.~Ferrari.
\newblock Optimal boundary surface for irreversible investment with stochastic
  costs.
\newblock {\em Math. Oper. Res.}, 42(4):1135--1161, 2017.

\bibitem{DeAPe}
T.~De~Angelis and G.~Peskir.
\newblock Global {$C^1$} regularity of the value function in optimal stopping
  problems.
\newblock {\em Ann. Appl. Probab.}, 30(3):1007--1031, 2020.

\bibitem{LNP}
C.~Donati-Martin, M.~\'Emery, A.~Rouault, and C.~Stricker, editors.
\newblock {\em S\'{e}m.\ {P}robab. {XL}}, volume 1899 of {\em Lecture Notes in
  Math.}
\newblock Springer, Berlin, 2007.

\bibitem{eisenbaum2000integration}
N.~Eisenbaum.
\newblock Integration with respect to local time.
\newblock {\em Potential Anal.}, 13(4):303--328, 2000.

\bibitem{eisenbaum2001ito}
N.~Eisenbaum.
\newblock On {I}t{\^o}’s formula of {F}{\"o}llmer and {P}rotter.
\newblock In {\em S{\'e}m.\ Probab.\ XXXV}, pages 390--395. Springer, 2001.

\bibitem{eisenbaum2006local}
N.~Eisenbaum.
\newblock Local time--space stochastic calculus for {L}{\'e}vy processes.
\newblock {\em Stoch.\ Process.\ Appl.}, 116(5):757--778, 2006.

\bibitem{ekstrom2020multi}
E.~Ekstr{\"o}m and Y.~Wang.
\newblock Multi-dimensional sequential testing and detection.
\newblock {\em arXiv:2009.02226}, 2020.

\bibitem{elworthy2007generalized}
K.~D. Elworthy, A.~Truman, and H.~Zhao.
\newblock Generalized {I}t\^{o} formulae and space-time {L}ebesgue-{S}tieltjes
  integrals of local times.
\newblock In {\em S\'{e}m.\ {P}robab.\ {XL}}, volume 1899 of {\em Lecture Notes
  in Math.}, pages 117--136. Springer, Berlin, 2007.

\bibitem{ernst}
P.~A. Ernst, G.~Peskir, and Q.~Zhou.
\newblock Optimal real-time detection of a drifting {B}rownian coordinate.
\newblock {\em Ann. Appl. Probab.}, 30(3):1032--1065, 2020.

\bibitem{errami2002ito}
M.~Errami, F.~Russo, and P.~Vallois.
\newblock It{\^o}'s formula for ${C}^{1,\lambda}$-functions of a c{\`a}dl{\`a}g
  process and related calculus.
\newblock {\em Probab.\ Theory Related Fields}, 122(2):191--221, 2002.

\bibitem{feng2007generalized}
C.~Feng, H.~Zhao, et~al.
\newblock A generalized {I}to's formula in two-dimensions and stochastic
  {L}ebesgue-{S}tieltjes integrals.
\newblock {\em Electron.\ J.\ Probab.}, 12:1568--1599, 2007.

\bibitem{Fe2018}
G.~Ferrari.
\newblock On the optimal management of public debt: a singular stochastic
  control problem.
\newblock {\em SIAM J. Control Optim.}, 56(3):1938--1975, 2018.

\bibitem{fleming}
W.~H. Fleming and H.~M. Soner.
\newblock {\em Controlled {M}arkov processes and viscosity solutions},
  volume~25 of {\em Stochastic Modelling and Applied Probability}.
\newblock Springer, New York, second edition, 2006.

\bibitem{follmer2000ito}
H.~F{\"o}llmer and P.~Protter.
\newblock On {I}t{\^o}'s formula for multidimensional {B}rownian motion.
\newblock {\em Probab.\ Theory Related Fields}, 116(1):1--20, 2000.

\bibitem{follmer1995quadratic}
H.~F{\"o}llmer, P.~Protter, and A.~N. Shiryayev.
\newblock Quadratic covariation and an extension of {I}t{\^o}'s formula.
\newblock {\em Bernoulli}, pages 149--169, 1995.

\bibitem{gapeev}
P.~V. Gapeev.
\newblock Bayesian switching multiple disorder problems.
\newblock {\em Math. Oper. Res.}, 41(3):1108--1124, 2016.

\bibitem{ghomrasni2004local}
R.~Ghomrasni and G.~Peskir.
\newblock Local time-space calculus and extensions of {I}t{\^o}’s formula.
\newblock In {\em High Dimensional Probability III}, pages 177--192. Springer,
  2004.

\bibitem{sd1992finite}
S.~Jacka and J.~Lynn.
\newblock Finite-horizon optimal stopping, obstacle problems and the shape of
  the continuation region.
\newblock {\em Stochastics}, 39(1):25--42, 1992.

\bibitem{johnson2017quickest}
P.~Johnson and G.~Peskir.
\newblock Quickest detection problems for {B}essel processes.
\newblock {\em Ann.\ Appl.\ Probab.}, 27(2):1003--1056, 2017.

\bibitem{KS}
I.~Karatzas and S.~E. Shreve.
\newblock {\em Brownian motion and stochastic calculus}, volume 113 of {\em
  Graduate Texts in Mathematics}.
\newblock Springer-Verlag, New York, second edition, 1991.

\bibitem{krylov2008controlled}
N.~V. Krylov.
\newblock {\em Controlled diffusion processes}, volume~14.
\newblock Springer Science \& Business Media, 2008.

\bibitem{Kyprianou}
A.~E. Kyprianou and B.~A. Surya.
\newblock A note on a change of variable formula with local time-space for
  {L}\'{e}vy processes of bounded variation.
\newblock In {\em S\'{e}m.\ {P}robab. {XL}}, volume 1899 of {\em Lecture Notes
  in Math.}, pages 97--104. Springer, Berlin, 2007.

\bibitem{moret2001generalization}
S.~Moret and D.~Nualart.
\newblock Generalization of {I}t{\^o}'s formula for smooth nondegenerate
  martingales.
\newblock {\em Stoch.\ Process.\ Appl.}, 91(1):115--149, 2001.

\bibitem{peskir2005change}
G.~Peskir.
\newblock A change-of-variable formula with local time on curves.
\newblock {\em J.\ Theoret.\ Probab.}, 18(3):499--535, 2005.

\bibitem{peskir2007change}
G.~Peskir.
\newblock A change-of-variable formula with local time on surfaces.
\newblock In {\em S\'{e}m.\ {P}robab.\ {XL}}, volume 1899 of {\em Lecture Notes
  in Math.}, pages 69--96. Springer, Berlin, 2007.

\bibitem{peskir2019continuity}
G.~Peskir.
\newblock Continuity of the optimal stopping boundary for two-dimensional
  diffusions.
\newblock {\em Ann. Appl. Probab.}, 29(1):505--530, 2019.

\bibitem{PSbook}
G.~Peskir and A.~Shiryaev.
\newblock {\em Optimal stopping and free-boundary problems}.
\newblock Lectures in Mathematics ETH Z\"{u}rich. Birkh\"{a}user Verlag, Basel,
  2006.

\bibitem{pham2009continuous}
H.~Pham.
\newblock {\em Continuous-time stochastic control and optimization with
  financial applications}, volume~61.
\newblock Springer Science \& Business Media, 2009.

\bibitem{protterbook}
P.~E. Protter.
\newblock {\em Stochastic integration and differential equations}, volume~21 of
  {\em Stochastic Modelling and Applied Probability}.
\newblock Springer-Verlag, Berlin, 2005.

\bibitem{rozkosz1996stochastic}
A.~Rozkosz.
\newblock Stochastic representation of diffusions corresponding to divergence
  form operators.
\newblock {\em Stoch.\ Process.\ Appl.}, 63(1):11--33, 1996.

\bibitem{russo1996ito}
F.~Russo and P.~Vallois.
\newblock It{\^o} formula for ${C}^1$-functions of semimartingales.
\newblock {\em Probab.\ Theory Related Fields}, 104(1):27--41, 1996.

\bibitem{SS89}
H.~Soner and S.~Shreve.
\newblock Regularity of the value function for a two-dimensional singular
  stochastic control problem.
\newblock {\em SIAM J. Control Optim.}, 27(4):876--907, 1989.

\bibitem{SS91}
H.~Soner and S.~Shreve.
\newblock A free boundary problem related to singular stochastic control: the
  parabolic case.
\newblock {\em Comm. Partial Differential Equations}, 16(2-3):373--424, 1991.

\bibitem{W18}
D.~Wilson.
\newblock The local time-space integral and stochastic differential equations.
\newblock {\em arXiv:1812.07566}, 2018.

\bibitem{W19}
D.~Wilson.
\newblock Change of variables with local time on surfaces for jump processes.
\newblock {\em arXiv:1901.04039}, 2019.

\end{thebibliography}

\end{document}